\documentclass[intlimits,a4paper,10pt,reqno]{amsproc}
\usepackage{amsmath,color}
\usepackage{amssymb}
\usepackage{amsfonts}
\usepackage{amsxtra}
\usepackage{graphicx}
\usepackage{xspace}
\usepackage{multirow}
\usepackage{enumitem}

\usepackage[active]{srcltx}

\definecolor{rltred}{rgb}{0.75,0,0}
\definecolor{rltgreen}{rgb}{0,0.5,0}
\definecolor{rltblue}{rgb}{0,0,0.75}

\usepackage[%
   colorlinks=true,
   urlcolor=rltblue,       
   filecolor=rltgreen,     
   citecolor=rltgreen,     
   linkcolor=rltred,       
   bookmarks=true,
   unicode,
   plainpages=false,
   ]{hyperref}

\usepackage[operator,rose,frak]
{paper_diening}
\usepackage{amsmath}
\usepackage{amssymb,amsthm,amsxtra,amscd}

\newcommand{\Oe}{\widetilde \Omega_\vep}

\newcommand{\td}{\partial_{\tau}}
\newcommand{\tn}{\partial_{3}}
\renewcommand{\O}{\Omega}

\newcommand{\eR}{{\mathbb R}}

\newcommand{\T}{\mathbf{S}}
\newcommand{\F}{\mathbf{F}}

\newcommand{\Bog}{{\tt B}}

\DeclareMathOperator{\diver}{div}

\DeclareMathOperator{\spt}{spt}

\newcommand{\intO}{\int_\O}

\newcommand{\para}{{\delta}}

\newcommand{\dist}{\operatorname{dist}}

\newtheorem{theorem}[equation]{Theorem}
\newtheorem{lemma}[equation]{Lemma}
\newtheorem{propo}[equation]{Proposition}
\newtheorem{Lemma}[equation]{Lemma}
\newtheorem{corollary}[equation]{Corollary}
\newtheorem{Proposition}[equation]{Proposition}
\newtheorem{definition}[equation]{Definition}

\newtheorem{remark}[equation]{Remark}

\numberwithin{equation}{section}

\newcommand{\otimess}{\overset{s}{\otimes}}

\newcommand{\trapm}[1]{{#1}_{\pm\tau}}
\newcommand{\difpm}[1]{d^\pm{#1}}
\newcommand{\trap}[1]{{#1}_{\tau}}
\newcommand{\difp}[1]{d^+{#1}}
\newcommand{\tran}[1]{{#1}_{-\tau}}
\newcommand{\difn}[1]{d^-{#1}}
\newcommand{\II}{\Lambda}
\newcommand{\bou}{\partial\Omega}
\newcommand{\meantmp}[2]{#1\langle{#2}#1\rangle}
\newcommand{\mean}[1]{\meantmp{}{#1}}
\newcommand{\bigmean}[1]{\meantmp{\big}{#1}}

\begin{document}

\title{Global regularity properties of steady shear thinning flows}
\author{Luigi C. Berselli }
\address{Luigi C. Berselli, Dipartimento di Matematica, Universit\`a di Pisa, Via
  F.~Buonarroti~1/c, I-56125 Pisa, ITALY.}
\email{luigi.carlo.berselli@unipi.it} 
\author {Michael  R\r u\v zi\v cka{}} 
\address{Michael  R\r u\v zi\v cka{}, Institute of Applied Mathematics,
   Albert-Ludwigs-University Freiburg, Eckerstr.~1, D-79104 Freiburg,
   GERMANY.}
\email{rose@mathematik.uni-freiburg.de}

\begin{abstract}
  In this paper we study the regularity of weak solutions to systems
  of $p$-Stokes type, describing the motion of some shear thinning
  fluids in certain steady regimes.  In particular we address the
  problem of regularity up to the boundary improving previous results
  especially in terms of the allowed range for the parameter $p$.
  \\[3mm]
\end{abstract}

\maketitle

%
%
\section{Introduction}
\label{sec:intro}
%
%
%
In this paper we study regularity for weak solution the steady Stokes
approximation for flows of shear thinning fluids which is given by
\begin{equation}
  \label{eq}
  \begin{aligned}
    -\divo \bfS(\bfD\bu)+\nabla \pi&=\bff\qquad&&\text{in
    }\Omega,
    \\
    \divo\bu&=0\qquad&&\text{in }\Omega,
    \\
    \bfu &= \bfzero &&\text{on } \partial \Omega.
  \end{aligned}
\end{equation}
Here $\Omega\subset\setR^3$ is a bounded domain with a $C^{2,1}$
boundary $\partial\Omega$ (we restrict ourselves to the most interesting problem
from the physical point of view, even if results can be easily
transferred to the problem in $\setR^d$ for all $d\geq 2$). The
unknowns are the velocity vector field $\bu=(u^1,u^2,u^3)^\top$ and
the scalar pressure $\pi$, while the external body force
$\bff=(f^1,f^2,f^3)^\top$ is given. The extra stress tensor $\bfS$
depends only on $\bfD\bu:=\tfrac 12(\nabla\bu+\nabla\bu^\top)$, the
symmetric part of the velocity gradient $\nabla \bu$. Physical
interpretation and discussion of some non-Newtonian fluid models can
be found, e.g., in~\cite{bird, mrr, ma-ra-model}. The relevant example
which we will study is the following
\begin{equation}
  \label{eq:stress}
  \bfS(\bD \bu) = \mu_0 \bD\bu + \mu_1 (\delta+\abs{\bD \bu})^{p-2}
  \bD\bu\,, 
\end{equation}
with $p \in (1,2]$, $\delta>0$, and $\mu_0, \mu_1 \ge 0$ satisfying
$\mu_1 >0$.

We study global regularity properties of second order derivatives of
weak solutions to~\eqref{eq} for sufficiently smooth bounded domains
$\O$. If $\mu_0 >0$ we obtain the optimal result, namely
$\bF(\bD \bu) \in W^{1,2}(\Omega)$, where the nonlinear tensor-valued
function $\bF$ is defined in~\eqref{eq:def_F}. This is the same result
as for the $p$-Stokes problem in the periodic setting.  For $\mu_0=0$
we prove (among other results) that $\bF(\bD\bu) \in W^{1,q}(\O)$, for
some $q=q(p)\in[1,2[$. The precise results are formulated in
Theorem~\ref{thm:MT1} and Theorem~\ref{thm:MT2}, where we also
write the regularity results in terms of Sobolev spaces. %
We treat here the case without convective term, since this quantity
can be handled in a more or less standard way, by Sobolev estimates,
once the precise regularity in terms of the right-hand side is proved.

Our main interest is to handle the case of a non-flat boundary and to
consider the full range $p\in]1,2[$. We recall that regularity results
in the flat have been obtained in~\cite{Crispo-2009} and under
various conditions by in~\cite{hugo-thin} (respectively
in the case $\mu_0>0$ and $\mu_0=0$). The situation in the non-flat
case becomes much more technical and we refer to the paper~\cite{mnr}
for an early treatment in the case $p>2$. Further results in the case
$p>2$ are present in~\cite{hugo-halfspace,hugo-nonflat,hugo-petr-rose}
while some results in the case $p<2$ are given
in~\cite{hugo-thin-nonflat}.  For a similar problem without pressure,
strong results are proved in~\cite{SS00} in the flat case.
For a treatment of simpler problems in the non flat case (nonlinear
elliptic problems without the divergence constraint)
see~\cite{Ber-Gri2-2016}.

In all previous studies of the $p$-Stokes problem for $p<2$ a
technical restriction $p> \frac 32$ occurs, which is due to the
presence of some algebraic systems to recover certain derivatives in
the normal direction. We are now able to remove this restriction by
deriving an algebraic system for a more intrinsic quantity related to
the stress tensor. Here we are giving a self-contained treatment of
the steady problem in the whole range $p \in ]1,2[$, and we also recover all
previously known results in the range $\frac{3}{2}<p<2$.  We point out
that the main difficulty is that of treating at the same time the
difficulties arising from: i) a non-linear stress tensor; ii) the
divergence-free condition (with the related pressure); iii) a
non-flat domain. 

\medskip

\noindent\textbf{Plan of the paper} The paper is organized as
follows. In Section~\ref{sec:prep} we recall the notation used
throughout the paper. Moreover, we recall some basic facts related to
the difference quotient in tangential directions and to the extra
stress tensor $\T$. In Section~\ref{sec:Stokes} we prove
Theorem~\ref{thm:MT1}. In particular, we treat in detail the
regularity in tangential directions in Section~\ref{sec:tan} and in
normal directions in Section~\ref{sec:nor}. Moreover, we prove some
regularity properties of the pressure. The same procedure is carried
out for the proof of Theorem~\ref{thm:MT2} in
Section~\ref{sec:p-stokes}. 
%
%
%
\section{Preliminaries and main results}
\label{sec:prep} 
In this section we introduce the notation we will use, state the
precise assumptions on the extra stress tensor $\bS$, and formulate
the main results of the paper. 
\subsection{Function spaces}
We use $c, C$ to denote generic constants, which may change from line
to line, but are independent of the crucial quantities. Moreover, we
write $f\sim g$ if and only if there exists constants $c,C>0$ such
that $c\, f \le g\le C\, f$.

In addition to the classical space
$(C^{k,\lambda}(\Omega),\,\|\,.\,\|_{C^{k,\lambda}})$ of H\"older
continuous functions for $0<\lambda<1$ (and Lipschitz-continuous when
$\lambda=1$) we use standard Lebesgue spaces
$(L^p(\Omega),\,\|\,.\,\|_{p})$ and Sobolev spaces
$(W^{k,p}(\Omega),\,\|\,.\,\|_{k,p})$, where $\Omega \subset \setR^3$,
is a sufficiently smooth bounded domain. When dealing with functions
defined only on some open subset $\omega\subset \Omega$, we denote the
norm in $L^p(\omega)$ by $\|\,.\,\|_{p,\omega}$. The symbol $\spt f$ denotes the
support of the function $f$. We do not distinguish between scalar,
vector-valued or tensor-valued function spaces. However, we denote
vectors by boldface lower-case letter as e.g.~$\bu$ and tensors by
boldface upper case letters as e.g.~$\bS$. If $\bu\in \eR^m$ and
$\bv\in\eR^n$ then the tensor product
$\bu\otimes \bv\in \eR^{m \times n}$ is defined as
$(\bu\otimes \bv)_{ij}:=u_iv_j$. If $m=n$ then
$\bu\otimess \bv:=\frac 12 (\bu\otimes \bv + (\bu\otimes
\bv)^\top)$.
The Euclidean scalar product is denoted by $\bu\cdot \bv= u_i v_i$ and
the scalar product $\bA\cdot\bB:=A_{i j}B_{i j}$ of second order
tensor is denoted also by the same symbol. Here and in the sequel we
use the summation convention over repeated Latin indices. %
The space $W^{1,p}_0(\Omega)$ is the closure of the compactly
supported, smooth functions $C^\infty_0 (\Omega)$ in
$W^{1,p}(\Omega)$. Thanks to the \Poincare{} inequality we equip
$W^{1,p}_0(\Omega)$ with the gradient norm $\norm{\nabla\,
  \cdot\,}_p$. We denote by $W^{1,p}_{0,\divo}(\O)$ the subspace of
$W^{1,p}_{0}(\O)$ consisting of divergence-free vector fields $\bu$,
i.e., such that $\divo \bu =0$. We denote by $\abs{M}$ the
$3$-dimensional Lebesgue measure of a measurable set $M$. The mean
value of a locally integrable function $f$ over a measurable set $M
\subset \Omega$ is denoted by $\mean{f}_M:= \dashint_M f \, d\bx =\frac
1 {|M|}\int_M f \, d\bx$. By $L^p_0(\Omega)$ we denote the subspace of
$L^p(\Omega)$ consisting of functions $f$ with vanishing mean value,
i.e., $\mean {f}_{\O} =0$. For a normed space $X$ we denote its
topological dual space by $X^*$.

We will also use Orlicz and Sobolev--Orlicz spaces
(cf.~\cite{ren-rao}). We use N-functions
$\psi \,:\, \setR^{\geq 0} \to \setR^{\geq 0}$. We always assume that
$\psi$ and the conjugate N-function $\psi ^*$ satisfy the
$\Delta_2$-condition. We denote the smallest constant such that
$\psi(2\,t) \leq K\, \psi(t)$ by $\Delta_2(\psi)$.
We denote by $L^\psi(\Omega)$ and $W^{1,\psi}(\Omega)$ the classical
Orlicz and Sobolev-Orlicz spaces, i.e, $f \in L^\psi(\Omega)$ if the
\textit{modular} $\rho_\psi(f):=\int_\Omega \psi(\abs{f})\, d\bx $ is
finite and $f \in W^{1,\psi}(\Omega)$ if $f$ and $ \nabla f$ belong to
$L^\psi(\Omega)$.  When equipped with the Luxembourg norm $\norm
{f}_{\psi}:= \inf \set{\lambda >0 \fdg \int_\Omega
  \psi(\abs{f}/\lambda)\, d\bx \le 1}$ the space $L^\psi(\Omega)$
becomes a Banach space. The same holds for the space
$W^{1,\psi}(\Omega)$ if it is equipped with the norm $\norm {\cdot
}_{\psi} +\norm {\nabla \cdot}_{\psi} $.  Note that the dual space
$(L^\psi(\Omega))^*$ can be identified with the space
$L^{\psi^*}(\Omega)$. By $W^{1,\psi}_0(\Omega)$ we denote the closure
of $C^\infty_0(\Omega)$ in $W^{1,\psi}(\Omega)$ and equip it with the
gradient norm $\norm{\nabla \cdot}_\psi$. By $L^\psi_0(\Omega)$ and
$C^\infty_{0,0}(\Omega)$ we denote the subspace of $L^\psi(\Omega)$
and $C^\infty_0(\Omega)$, respectively, consisting of functions $f$
such that $\mean {f}_{\O}
=0$. 

We need the following refined version of the Young inequality: for all
$\varepsilon >0$ there exists $c_\epsilon>0 $, depending only on
$\Delta_2(\psi),\Delta_2( \psi ^*)<\infty$, such that for all
$s,t\geq0$ it holds
\begin{align}
  \label{eq:ineq:young}
  \begin{split}
    ts&\leq \epsilon \, \psi(t)+ c_\epsilon \,\psi^*(s)\,,
    \\
    t\, \psi'(s) + \psi'(t)\, s &\le \epsilon \, \psi(t)+ c_\epsilon
    \,\psi(s)\,.
  \end{split}
\end{align}
\subsection{$(p,\delta)$-structure}
\label{sec:stress_tensor}
We now define what it means that a tensor field $\bS$ has
$(p,\delta)$-structure.  A detailed discussion and full proofs of the
results cited can be found in~\cite{die-ett,dr-nafsa}. For a tensor
$\bfP \in \setR^{3 \times 3} $ we denote its symmetric part by
$\bP^\sym:= \frac 12 (\bP +\bP^\top) \in \setR_\sym ^{3 \times 3}:=
\set {\bfP \in \setR^{3 \times 3} \,|\, \bP =\bP^\top}$. 
We use the notation $\abs{\bP}^2=\bP \cdot \bP ^\top$.

It is convenient to define for $t\geq0 $ a special N-function
$\phi(\cdot)=\phi(p,\delta;\cdot)$, for $p \in (1,\infty)$, $\delta\ge
0$, by
\begin{equation*} 
  \label{eq:5}
  \varphi(t):= \int _0^t \gamma(s)\, ds\qquad\text{with}\quad
  \gamma(t) := (\delta +t)^{p-2} t\,.
\end{equation*}
The function $\phi$ satisfies, uniformly in $t$ and independent of
$\delta$, the important equivalence\footnote{Note that if $\phi''(0)$
  does not exist, the left-hand side in~\eqref{eq:equi1} is
  continuously extended by zero for $t= 0$. }
\begin{align}
  \label{eq:equi1}
  \phi''(t)\, t \sim \phi'(t)
\end{align}
since 
\begin{equation}
  \min\set{1,p-1}\,(\para+t)^{p-2} \le \varphi''(t)\leq
  \max\set{1,p-1}(\para+t)^{p-2}\,.\label{eq:phi''}
\end{equation}
Moreover, the function $\phi$
satisfies the $\Delta_2$-condition with $\Delta_2(\phi) \leq
c\,2^{\max \set{2,p}}$ (hence independent of $\delta$). This implies
that, uniformly in $t$ and independent of $\delta$, we have
\begin{align}
  \label{eq:equi2}
  \phi'(t)\, t \sim \phi(t)\,.
\end{align}
The conjugate function $\phi^*$ satisfies $\phi^*(t) \sim
(\delta^{p-1} + t)^{p'-2} t^2$ where
$\frac{1}{p}+\frac{1}{p'}=1$. Also $\phi^*$ satisfies the
$\Delta_2$-condition with $\Delta_2(\phi^*) \leq c\,2^{\max
  \set{2,p'}}$. Using the structure of $\varphi$ we get for all $t,
\lambda \ge 0$
\begin{equation}
  \label{eq:est-l}
  \varphi (\lambda\, t) \le \max (\lambda^p,\lambda^2) \varphi (t) \le
  \max (1,\lambda^2) \varphi (t)  
\end{equation}
if $p\le 2$.

For a given N-function $\psi$ we define the {\rm shifted N-functions}
$\set{\psi_a}_{a \ge 0}$, cf. ~\cite{die-ett,die-kreu,dr-nafsa}, for
$t\geq0$ by
\begin{align}
  \label{eq:phi_shifted}
  \psi_a(t):= \int _0^t \psi_a'(s)\, ds\qquad\text{with }\quad
  \psi'_a(t):=\psi'(a+t)\frac {t}{a+t}.
\end{align}
\begin{remark} \label{rem:phi_a} 
{\rm 
  
  (i) Defining $\omega(t)=\omega(q;t):=\frac 1q t^q$, $q\in
  (1,\infty)$ we have for the above defined N-function 
  \begin{equation*}
     \phi(t)= \omega_\delta(p;t).
   \end{equation*}

  (ii) 
  Note that $\phi_a(t) \sim (\delta+a+t)^{p-2} t^2$ and also
  ${(\phi_a)^*(t) \sim ((\delta+a)^{p-1} + t)^{p'-2} t^2}$.  The
  families $\set{\phi_a}_{a \ge 0}$ and $\set{(\phi_a)^*}_{a \ge 0}$
  satisfy the $\Delta_2$-condition uniformly with respect to ${a \ge
    0}$, with $\Delta_2(\phi_a) \leq c\, 2^{\max \set{2,p}}$ and
  $\Delta_2((\phi_a)^*) \leq c\, 2^{\max \set{2,p'}}$,
  respectively. The equivalences~\eqref{eq:equi1} and~\eqref{eq:equi2}
  are satisfied for the families $\set{\phi_a}_{a \ge 0}$ and
  $\set{(\phi_a)^*}_{a \ge 0}$, uniformly in $a \ge 0$.  }
\end{remark}
%
\begin{definition}[$(p,\delta)$-structure]
\label{ass:1}
  We say that a tensor field $\bS\colon \setR^{3 \times 3} \to
  \setR^{3 \times 3}_\sym $ belonging to $C^0(\setR^{3 \times
    3},\setR^{3 \times 3}_\sym )\cap C^1(\setR^{3 \times 3}\setminus
  \{\bfzero\}, \setR^{3 \times 3}_\sym ) $, satisfying $\bS(\bP) =
  \bS\big (\bP^\sym \big )$, and $\bS(\mathbf 0)=\mathbf 0$ possesses
  {\rm $(p,\para)$-structure}, if for some $p \in (1, \infty)$,
  $\para\in [0,\infty)$, and the N-function
  $\varphi=\varphi_{p,\delta}$ (cf.~\eqref{eq:5}) there exist
  constants $\kappa_0, \kappa_1 >0$ such that
   \begin{equation}
     \label{eq:ass_S}
     \begin{aligned}
       \sum\limits_{i,j,k,l=1}^3 \partial_{kl} S_{ij} (\bP)
       Q_{ij}Q_{kl} &\ge \kappa_0
       \,\phi''(|\bP^\sym|)|\bQ^\sym|^2\,,
       \\[-2mm]
       \big |\partial_{kl} S_{ij}({\bP})\big | &\le \kappa_1 \,\phi''
       (|\bP^\sym|)\,.
     \end{aligned}
   \end{equation} 
   are satisfied for all $\bP,\bQ \in \setR^{3\times 3} $ with
   $\bP^\sym \neq \bfzero$.  The constants $\kappa_0$, $\kappa_1$, and
   $p$ are called the {\em characteristics} of $\bfS$.
\end{definition}

\begin{remark}
  The above assumption is motivated by the typical examples for the
  extra stress tensor in mathematical fluid mechanics. For example
  constitutive relations of  power-law type, Carreau type, Cross-type
  or~\eqref{eq:stress} satisfy this assumption. 
  We refer the reader to~\cite{bdr-7-5,dr-7-5,mnrr,mrr} for a more
  detailed discussion leading to Definition~\ref {ass:1}.
\end{remark}

\begin{remark}
  \label{rem:delta0}
  {\rm (i) Assume that $\bS$ has $(p,\delta)$-structure for some
    $\delta \in [0,\delta_0]$. Then, if not otherwise stated, the
    constants in the estimates depend only on the characteristics of
    $\bfS$ and on $\delta_0$ but are independent of $\delta$. This
    dependence comes from the difference between the modular and the
    norm in the case of Orlicz spaces.
 
    (ii) An important example of an extra stress $\bfS$ having
     $(p,\delta)$-structure is given by $ \bfS(\bfP) =
    \phi'(\abs{\bfP^\sym})\abs{\bfP^\sym}^{-1} \bfP^\sym$.  In this
    case the characteristics of $\bfS$, namely $\kappa_0$, $\kappa_1$, and $p$,
    depend only on $p$ and are independent of $\delta \geq 0$.
  }
\end{remark}

\begin{remark}
  \label{rem:equi-norm}
  {\rm 
    For the family $(\phi _a)$, with $a \in [0,a_0]$, ${\delta
    \in [0,\delta_0]}$ and $p \in (1,\infty)$, we get $L^{\phi_a^*}(\Omega) =
    L^{p'}(\Omega)$ and $W^{1,\phi_a}(\Omega)=W^{1,p}(\Omega) $ with
    uniform equivalence of the corresponding norms depending on $p$, $a_0$
    and $\delta_0$, since $\Omega$ is bounded and $\phi_a$ and
    $\omega(p;\cdot)$ are equivalent at infinity (cf.~\cite{ren-rao}).
  } 
\end{remark}

To a tensor field $\bS$ with $(p,\delta)$-structure we associate
the tensor field $\bF\colon\setR^{3 \times 3} \to \setR^{3 \times
  3}_\sym$ defined through
\begin{align}
  \label{eq:def_F}
  \bF(\bP):= \big (\para+\abs{\bP^\sym} \big )^{\frac
    {p-2}{2}}{\bP^\sym } \,.
\end{align}

The connection between $\bfS$, $\bfF$, and $\set{\phi_a}_{a \geq 0}$
is best explained by the following proposition
(cf.~\cite{die-ett,dr-nafsa}). 
\begin{propo}
  \label{lem:hammer}
  Let $\bfS$ has $(p,\delta)$-structure, and let $\bfF$ be defined
  in~\eqref{eq:def_F}.  Then
  \begin{subequations}
    \label{eq:hammer}
    \begin{align}
      \label{eq:hammera}
      \big({\bfS}(\bfP) - {\bfS}(\bfQ)\big) \cdot \big(\bfP-\bfQ \big)
      &\sim \bigabs{ \bfF(\bfP) - \bfF(\bfQ)}^2
      \\
      &\sim \phi_{\abs{\bfP^\sym}}(\abs{\bfP^\sym - \bfQ^\sym})\,,
      \label{eq:hammerb}
      \\
      \label{eq:hammerc}
      &\sim \phi''\big( \abs{\bfP^\sym} + \abs{\bfP^\sym - \bfQ^\sym}
      \big)\abs{\bfP^\sym - \bfQ^\sym}^2
      \\
      \abs{\bfS(\bfP) - \bfS(\bfQ)} &\sim
      \phi'_{\abs{\bfP^\sym}}\big(\abs{\bfP^\sym -
        \bfQ^\sym}\big)\,.   \label{eq:hammere}  
      \intertext{uniformly in $\bfP, \bfQ \in \setR^{3 \times 3}$.
        Moreover,  uniformly in $\bfQ \in \setR^{3 \times 3}$,} 
      \label{eq:hammerd}
      \bfS(\bfQ) \cdot \bfQ &\sim \abs{\bfF(\bfQ)}^2 \sim
      \phi(\abs{\bfQ^\sym}).
    \end{align}
  \end{subequations}
  The constants depend only on the characteristics of $\bfS$.
\end{propo} 
For a detailed discussion of the
properties of $\T$ and $\F$ and their relation to Orlicz spaces and
N-functions we refer the reader to~\cite{dr-nafsa},~\cite{bdr-7-5}. We
just want to mention that
\begin{equation}
  \label{eq:hammerq}
  \abs{\bfF(\bfQ)} +\delta^{\frac p2} \sim \abs{\bfQ^\sym}^{\frac p2}
  +\delta^{\frac p2},
\end{equation}
with constants depending only on $p$.
\begin{remark}[Natural distance]
  \label{rem:natural_dist}
  {\rm  
    In view of the previous lemma we have, for all $\bfu, \bfw \in
  W^{1,\phi}(\Omega)$,
  \begin{align*}
    \skp{\bfS(\bD \bfu) \!-\! \bfS(\bD\bfw)}{\bD\bfu \!-\! \bD \bfw}
    &\sim 
    \norm{\bfF(\bD\bfu) \!-\! \bfF(\bD\bfw)}_2^2 \,\sim \int_\Omega\!
    \phi_{\abs{\bD\bfu}}(\abs{\bD\bfu \!-\! \bD\bfw}) \, d\bx.
  \end{align*}
  The constants depend only on the characteristics of $\bfS$.  The
  last expression equals the quasi-norm introduced in~\cite{BL1994b}
  raised to the power $\rho = \max \set{p,2}$. We refer
  to all three equivalent quantities as the \textit{natural distance}.
}
\end{remark}
Since in the following we shall insert into $\T $ and $\F$ only
symmetric tensors, we can drop in the above formulas the superscript
``$^\sym $'' and restrict the admitted tensors to symmetric ones.
\subsection{Description and properties of the boundary}
\label{sec:bdr} 
We assume that the boundary $\partial\O$ is of class $C^{2,1}$, that
is for each point $P\in\partial\O$ there are local coordinates such
that in these coordinates we have $P=0$ and $\partial\O$ is locally
described by a $C^{2,1}$-function, i.e.,~there exist
$R_P,\,R'_P,\,r_P\in (0,\infty)$ and a $C^{2,1}$-function
$a_P:B_{R_P}^{2}(0)\to B_{R'_P}^1(0)$ such that
\begin{itemize}
\item   [\rm (b1) ] $\bx\in \partial\O\cap (B_{R_P}^{2}(0)\times
  B_{R'_P}^1(0))\ \Longleftrightarrow \ x_3=a_P(x_1,x_2)$,
\item   [\rm (b2) ] $\Omega_{P}:=\{(x,x_{3}): x=(x_1,x_2)
 \in  B_{R_P}^{2}(0),\ a_P(x)<x_3<a_P(x)+R'_P\}\subset \Omega$, 
\item [\rm (b3) ] $\nabla a_P(0)=\bfzero,\text{ and }\forall\,x=(x_1,x_2)
  \in B_{R_P}^{2}(0)\quad |\nabla a_P(x)|<r_P$,
\end{itemize}
where $B_{r}^k(0)$ denotes the $k$-dimensional open ball with center
$0$ and radius $r>0$.  By a slightly abuse of notation, to simplify
further formulas, we set $R_{P}:=\min\{R_{P},R'_{P}\}$.  Note also
that $r_P $ can be made arbitrarily small if we make $R_P$ small
enough.

In the sequel we will also use, for $0<\lambda<1$, the following
scaled open sets, $\lambda\, \Omega_P\subset \Omega_P$ defined as
follows
\begin{equation}
  \label{eq:scaled_omega_P}
  \lambda\, \Omega_P:=\{(x,x_{3}): x=(x_1,x_2)
 \in
  B_{\lambda R_P}^{2}(0),\ a_P(x)<x_3<a_P(x)+\lambda R_P\}.
\end{equation}
%
%
%
%
To prove our global estimates we first show local estimates near the
boundary in $\Omega_P$, for every $P\in\bou$. To this end we fix smooth
functions $\xi_{P}:\setR^{3}\to\setR$ such that $0\leq\xi_P\leq1$ and 
\begin{equation*}
 \hspace*{-30mm}{\rm (\ell 1)}\hspace{40mm}
 \chi_{\frac{1}{2}\Omega_P}(\bx)\leq\xi_P(\bx)\leq \chi_{\frac{3}{4}\Omega_P}(\bx),
\end{equation*}
where $\chi_{A}(\bx)$ is the indicator function of the measurable set
$A$. 
For the remaining interior estimate we also localize by a smooth function
${0\leq\xi_{00}\leq 1}$ such that $\spt \xi_{00}\subset\Omega_{00}$,
where $\Omega_{00}\subset \Omega$ is an open set such that
$\dist(\partial\Omega_{00},\,\partial\Omega)>0$.  
The local estimates near the boundary are obtained in two steps. In
the first one (see Sections~\ref{sec:tan} and~\ref{sec:tanp}) we
estimate in $\Omega_P$ only tangential derivatives as defined
below. In the second one we use the new obtained information and
compute the normal derivatives from the system. Since the boundary
$\bou $ is compact, we can use an appropriate finite covering of it
which, together with the interior estimate, yields the global estimate.

Let us now introduce the tangential derivatives near the boundary and
related concepts. To simplify the notation we fix $P\in \bou$,
$h\in (0,\frac{R_P}{16})$, and simply write $\xi:=\xi_P$, $a:=a_P$. We
use the standard notation $\bx =(x',x_3)^\top$ and denote by
$\be^i,i=1,2,3$ the canonical orthonormal basis in $\setR^3$. In the
following lower-case Greek letters take values $1,\, 2$.

A crucial technicality to handle non-flat boundaries is to define a
proper way of differentiation (and approximate partial derivatives) in
directions that are tangential to the boundary, at least in a
tubular neighborhood of $\partial\Omega$. For a function $g$ (when $g$
is vector or tensor valued the same is applied to all components) with
$\spt g\subset\spt\xi$ we define  positive and
negative tangential translations:
\begin{equation*}
\begin{aligned}
  \trap g(x',x_3)&:=g\big (x' +
  h\,\be^\alpha,x_3+a(x'+h\,\be^\alpha)-a(x')\big )\,,
  \\
  \tran g(x',x_3)&:=g\big (x' - h\,\be^\alpha,x_3+a(x' -
  h\,\be^\alpha)-a(x')\big )\,;
\end{aligned}
\end{equation*}
tangential differences
\begin{equation*}
  \Delta^+ g:=\trap g-g,\qquad\Delta^- g:=\tran g-g\,;
\end{equation*}
and tangential divided differences
\begin{equation*}
  \difp g:= h^{-1}\Delta^+ g\,, \qquad \difn g:=h^{-1} \Delta^- g\,.
\end{equation*}
It holds
\begin{align} 
  \label{eq:1}
  \difp g \to\td g :=\partial_\alpha g +\partial_\alpha
  a\, \partial_3 g  \qquad \text{ as } h\to 0,
\end{align} 
almost everywhere in $\spt\xi$, if $ g \in W^{1,1}(\O)$
(cf.~\cite[Sec.~3]{MNR2001}). Moreover, we have for all %
 $1<q<\infty$, $ g \in W^{1,q}(\O)$ and all sufficiently small
$h>0$, that 
\begin{align}
  \label{eq:2}
  \exists\,c(a)>0:\quad \norm{\difp g }_{q,\spt\xi} \le
  c(a)\norm{\nabla g }_q .
\end{align}
Conversely, if $\norm{\difp g  }_{q,\spt\xi} \le C$ for all
sufficiently small $h>0$, then 
\begin{align}
  \label{eq:2a}
  \norm {\partial _\tau  g }_{q,\spt\xi} \le C.
\end{align}

Now we formulate some auxiliary lemmas related to these objects. The
first lemma clarifies the fact that tangential translations and
tangential differences do not non commute with partial
derivatives. Also the explicit expressions can be used to
quantitatively estimate the so called commutation terms, as called in
turbulence theory~\cite{Ber-Gri-John2007}. For
simplicity we denote $\nabla a:=(\partial_1a,\partial_{2}a, 0)^\top$
and use the operations $\trap {(\cdot)}$, $\tran {(\cdot)}$,
$\Delta^+(\cdot) $, $\Delta^+(\cdot) $, $\difp {(\cdot)}$ and $\difn
{(\cdot)}$ also for vector-valued and tensor-valued functions,
intended as acting component-wise.
\begin{lemma}
  \label{lem:TD1} 
  Let $\bv\in W^{1,1}(\Omega)$ such that $\spt \bv
  \subset\spt\xi$. Then 
\begin{equation*}
\begin{aligned}
  \nabla\difpm \bv &=\difpm{\nabla \bv }+\trap{(\partial_3 \bv
    )}\otimes\difpm{\nabla a},
  \\
  \bD\difpm \bv &=\difpm{\bD \bv }+\trap{(\partial_3 \bv
    )}\otimess\difpm{\nabla a},
  \\
  \diver\difpm \bv &=\difpm\diver \bv +\trapm{(\partial_3 \bv
    )}\difpm{\nabla a}
  \\
  \nabla \bv _{\pm\tau} &= (\nabla \bv )_{\pm\tau} + \trapm{(\partial_3 \bv
    )}\difpm{\nabla a},
\end{aligned}
\end{equation*}
where $\otimess$ is defined component-wise also for scalar and tensor-valued functions.
\end{lemma}
The second lemma is devoted to the relation between tangential
differences and tangential translations, provided that $h$ is small
enough. 
\begin{lemma}
  \label{lem:TD2}
  Let $\spt g \subset\spt\xi$. Then
\begin{equation*}
  \trap{(\difn g )}=-\difp g ,\quad \tran{(\difp g )}=-\difn g , \quad 
  \difn  g_\tau = - \difp g .
\end{equation*}
\end{lemma}
The following variant of integration per parts will be often used.
\begin{lemma}
  \label{lem:TD3}
  Let $\spt g\cup\spt f\subset\spt\xi$ and $h$ small enough. Then
  \begin{equation*}
    \intO f\tran g \, d\bx =\intO\trap f g\, d\bx.
  \end{equation*}
  Consequently, $\intO f\difp g \, d\bx= \intO(\difn f )g\, d\bx$.
\end{lemma}
Also the following variant of the product rule will be used.
\begin{lemma}
  \label{lem:TD4}
  Let $\spt g\cup\spt f\subset\spt\xi$. Then
  \begin{equation*}
  \difpm (f g) = f_{\pm\tau} \,\difpm g + (\difpm f )\, g.
\end{equation*}
\end{lemma}
\medskip

If $\bS$ has $(p,\delta)$-structure we easily obtain from
Lemma~\ref{lem:hammer} the following equivalences
\begin{align}
  \begin{aligned}
    \label{eq:3}
    \abs{\difp \T (\bD\bu)} &\sim
    (\delta+|\bD\bu|+|\Delta^+{\bD\bu}|)^{p-2}|\difp \bD\bu|
    \\
    &\sim \phi''(|\bD\bu|+|\Delta^+{\bD\bu}|) |\difp \bD\bu|
    \\
    &\sim \big (\phi''(|\bD\bu|+|\Delta^+{\bD\bu}|) \big)^\frac
    12\abs{\difp \F (\bD\bu)}\,
    \\
    & \sim (\delta+|\bD\bu|+|\Delta^+{\bD\bu}|)^{\frac {p-2}2}|\difp
    \bF(\bD\bu)|\,,
  \end{aligned}
\end{align}    
\vspace{-3mm}
\begin{align}
  \hspace*{-9.7mm}
  \begin{aligned}
    \label{eq:3a}
    \difp \T (\bD\bu)\cdot \difp \bD\bu &\sim \abs{\difp \F (\bD\bu)}^2
    \\
    &\sim (\delta+|\bD\bu|+|\Delta^+\bD\bu|)^{p-2}|\difp \bD\bu|^2
    \\
    &\sim \phi''(|\bD\bu|+|\Delta^+{\bD\bu}|) |\difp \bD\bu|^2,
  \end{aligned}
\end{align}
with constants depending only on the characteristics of $\bS$ and $p$.
All assertions from this section may be proved by easy manipulations
of definitions and we drop their proofs.
\subsection{Main Results}
Now we can formulate our main results concerning the regularity
properties of weak solutions to problems~\eqref{eq}, with different
assumptions on the stress tensor. We especially focus on the two
different cases in which there is a part associated with the quadratic
growth or in which this is lacking.
\begin{theorem}
  \label{thm:MT1}
  Let $\T$ the extra stress tensor in~\eqref{eq} be given by
  $\bS=\bS^0 + \bS^1$, where $\bS^0$ satisfies Assumption~\ref{ass:1}
  with $p=2$ and $\bS^1$ satisfies Assumption~\ref{ass:1} for some
  $p\in (1,2)$, and $\delta \in (0,\infty)$.  Let $\bF$ be the
  associated tensor field to $\bS^1$. Let $\O\subset\eR^3$ be a
  bounded domain with boundary $\partial\Omega$ of class $C^{2,1}$,
  and let $\ff\in L^2(\O)$. Then, the unique weak solution $(\bu,\pi)\in
  W^{1,2}_{0,\divo}(\O)\times L^{2}_{0}(\O)$ of the problem ~\eqref{eq} satisfies
  \begin{equation*}
    \begin{split}
      \label{eq:est-reg-F1}
      \intO\abs{\nabla ^2 \bu}^2+\abs{\nabla
        \F(\bD\bu)}^2+\abs{\nabla\pi}^2 \, d\bx &\le c \,,
    \end{split}
  \end{equation*}  
  where the constant $c$ depends on $\|\ff\|_2$, on the characteristics
  of $\bS^0$, $\bS^1$, on $\delta^{-1}$, and on $\bou$. 
\end{theorem}
A more precise dependence of various quantities in terms of $\delta$
is given in the proof of the theorem.  Note that we obtain the same
regularity on $(\bu,\pi)$ as in the case of the (linear) Stokes
system, namely
\begin{equation*}
  \bu\in W^{2,2}(\Omega)\quad\text{and}\quad \pi\in
  W^{1,2}(\Omega). 
 \end{equation*}


Let us consider now the case in which there is only the nonlinear part
of the stress tensor.
\begin{theorem}
  \label{thm:MT2}
  Let the extra stress tensor $\T$ in~\eqref{eq} satisfy
  Assumption~\ref{ass:1} for some $p \in (1,2)$, and $\delta \in
  (0,\infty)$, and let $\bF$ be the associated tensor field to
  $\bS$. Let $\O\subset\eR^3$ be a bounded domain with $C^{2,1}$
  boundary and let $\ff\in L^{p'}(\O)$. Then, the unique weak solution
  $\bu\in W^{1,p}_{0,\divo}(\O)$ of the problem~\eqref{eq} satisfies
  \begin{equation*}
    \begin{aligned}
      \label{eq:est-reg-F2}
      \intO  \xi_{0}^2 \abs{\nabla \F(\bD\bu)}^2+
      \xi_{0}^2 \abs{\nabla\pi}^2 \, d\bx
      &\le c
      (\norm{\ff}_{p'},\norm{\xi_0}_{2,\infty},\delta^{-1}) \,, 
      \\
      \intO\xi^2 \abs{\td \F(\bD\bu)}^2+\xi^2 \abs{\td \pi}^2\,
      d\bx
      &\le c (\norm{\ff}_{p'},
      \norm{\xi}_{2,\infty},\norm{a}_{C^{2,1}},\delta^{-1}) \,,
      \\
      \intO\xi^{\frac {8p-4}{3p}} \abs{\tn \F(\bD\bu)}^{\frac
        {8p-4}{3p}}\! +\xi^{\frac {8p-4}{p}} \abs{\F(\bD\bu)}^{\frac
        {8p-4}{p}} \, d\bx 
      &\le c (\norm{\ff}_{p'},\norm{\xi}_{2,\infty},\norm{a}_{C^{2,1}},\delta^{-1}) \,,
      \\
      \int_\Omega \xi^{p} \abs{\tn \pi}^{\frac {4p-2}{p+1}} \, d\bx &\le c
      (\norm{\ff}_{p'},\norm{\xi}_{2,\infty},\norm{a}_{C^{2,1}},\delta^{-1}) \,.
    \end{aligned}
  \end{equation*}  
  Here $\xi_{0}$ is a cut-off function with support in the interior of
  $\Omega$, while $\xi=\xi_{P}$ is a cut-off function with support near to %
  the boundary $\partial \Omega$, as defined in Sec.~\ref{sec:bdr}.
The tangential derivative
  $\partial_\tau$ is defined locally in $\Omega_P$ by~\eqref{eq:1}. This in
  particular implies that $\bF(\bD\bu)\in L^{\frac{8p-4}{p}}(\Omega)$
  and, in terms of derivatives of $\bu$, that
  \begin{equation*}
    \label{eq:est-reg-sob}
      \xi \td \nabla \bu \in L^{\frac {8p-4}{3p}}(\Omega),\quad \xi
      \tn \nabla \bu \in L^{\frac {4p-2}{p+1}}(\Omega), \quad \nabla
      \bu \in L^{{4p-2}}(\Omega),\quad       \nabla ^{2}\bu\in L^{\frac {4p-2}{p+1}}(\Omega)\,,
  \end{equation*}
while the partial derivatives of the pressure satisfy
\begin{equation*}
  \xi\partial_{\tau}\pi\in L^{2}(\Omega),\quad \xi\partial_{3}\pi\in
  L^{{\frac {4p-2}{p+1}}}(\Omega),\quad \pi\in W^{1,{\frac {4p-2}{p+1}}}(\Omega).
\end{equation*}
\end{theorem}
We observe that the regularity results in
Theorems~\ref{thm:MT1}-\ref{thm:MT2} are the same as in the flat-case,
which are proved (for the restricted range $\frac{3}{2}<p<2$)
in~\cite{Crispo-2009} and~\cite{Ber2009a}, respectively. Nevertheless,
comprehensive analysis of the changes deriving from presence/absence
of the stress tensor with the quadratic part is given here with full
detail. This suggests that our results are quite sharp. (Note also
that in the flat case the estimates obtained are better behaved in
terms of $\delta$.) The main obstacle to the proof --and to possible
improvements-- of the above theorems is given by the coupling of the
boundary condition prescribed on the non-flat boundary $\partial \O$,
together with the solenoidal constraint, which results in the
appearance of the pressure term $\nabla\pi$ in the
equations~\eqref{eq}.

The case of systems where the extra stress tensor depends on the
symmetric velocity gradient, but without a solenoidal constraint, is
completely solved in the case of a flat boundary and $p<2$
in~\cite{SS00}.

\subsection{Auxiliary results}
Here we collect some auxiliary results needed in the sequel of the
paper. 
\begin{Lemma}
  \label{lem:quasinormtrick}
  For all $\epsilon>0$, there exists a constant $c_\epsilon>0$
  depending only on $\epsilon>0$ and the characteristics of $\bfS$
  such that for all sufficiently smooth vector fields $\bu$, $\bv$,
  and $\bw$ we have
  \begin{align*}
    &\skp{\bfS(\bD\bfu) - \bfS(\bD\bfv)}{\bD\bfw - \bD \bfv} \leq
    \epsilon\, \norm{\bfF(\bD\bfu) - \bfF(\bD\bfv)}_2^2 +c_\epsilon\,
    \norm{\bfF(\bD\bfw) - \bfF(\bD\bfv)}_2^2\,.
  \end{align*}
\end{Lemma}
\begin{proof}
  This is proved in~\cite[Lemma 2.3]{der}.
\end{proof}
\begin{Lemma}\label{lem:change}
  Let $\psi $ be an N--function satisfying the
  $\Delta_2$--condition. Then, for all $\bP,\bQ \in \setR^{n\times n}$
  and all $t\ge 0$ there holds
  \begin{align}
    \label{eq:33}
    \psi'_{\abs{\bP}}(t) &\le 2\Delta_2(\psi')\, \psi'_{\abs{\bQ}}(t) +
    \psi'_\abs{\bP}(\abs{\bP-\bQ}) \,,\\[1mm]
    \psi'_{\abs{\bP}}(t) &\le 2\Delta_2(\psi')\,\big ( \psi'_{\abs{\bQ}}(t) +
    \psi'_\abs{\bQ}(\abs{\bP-\bQ}) \big )\,.\label{eq:33a}
  \end{align}
\end{Lemma}
\begin{proof}
  This is proved in~\cite[Lemma 5.13, Remark 5.14]{dr-nafsa}.
\end{proof}
\begin{Lemma}[Change of shift]
  \label{lem:change2}
  Let $\psi$ be an N--function such that $\psi$ and $\psi^*$ satisfy
  the $\Delta_2$--condition.  Then for all $\delta \in (0,1)$ there
  exists $c_\vep=c_\vep (\Delta_2(\psi'))$ such that all $\bP,\bQ
  \in \setR^{n\times n}$, and all $t\ge 0$
  \begin{align}
    \label{eq:34c1}
    \psi_{\abs{\bP}}(t) &\le c_\vep\, \psi_{\abs{\bQ}}(t) +
    \vep\,\psi_\abs{\bP}\big    (\abs{\bP-\bQ}\big )\,\\[1mm]
    \label{eq:34c}
    \psi_{\abs{\bP}}(t)&\le c_\vep\,\psi_{\abs{\bQ}}(t) +
    \vep\,\psi_\abs{\bQ}(\abs{\bP-\bQ}) \,.
  \end{align}
  \begin{align}
    \label{eq:35c1}
    \big (\psi_{\abs{\bP}}\big )^*(t) &\le c_\vep\, \big
    (\psi_{\abs{\bQ}}\big )^*(t) +
    \vep\,\psi_\abs{\bP}\big    (\abs{\bP-\bQ}\big )\,\\[1mm]
    \label{eq:35c}
    \big (\psi_{\abs{\bP}}\big )^*(t)&\le c_\vep\,\big
    (\psi_{\abs{\bQ}}\big )^*(t) +
    \vep\,\psi_\abs{\bQ}(\abs{\bP-\bQ}) \,.
  \end{align}
\end{Lemma}
\begin{proof}
  This is proved in~\cite[Lemma 5.15, Lemma 5.18]{dr-nafsa}.
\end{proof}
%
%
\begin{Proposition}[Divergence equation in Orlicz spaces]
  \label{thm:bogovskii}
  Let $G \subset \setR^n$ be a \linebreak bounded John domain.  Then,
  there exists a linear operator $\Bog\,:\,C^\infty_{0,0}(G) \to
  C^\infty_0(G)$ which extends uniquely for all N-functions $\psi$
  with $\Delta_2(\psi), \Delta_2(\psi^*)<\infty$ to an operator
  $\Bog:\,L^\psi_0(G)\to W^{1,\psi}_0(G)$, satisfying $\divergence \Bog
  f = f$, and
  \begin{align*}
    \norm{\nabla \Bog f}_{L^\psi(G)} &\leq c\,
    \norm{f}_{L^\psi_0(G)},
    \\
    \int_G \psi\big(\abs{\nabla \Bog f}\big)\, d\bx &\leq c\,
    \int_G \psi\big(\abs{f}\big)\, d\bx.
  \end{align*}
  The constant~$c$ depends on $\Delta_2(\psi)$, $\Delta_2(\psi^*)$,
  and the John constant of $G$.
\end{Proposition}
\begin{proof}
  This is Theorem~4.2 in~\cite{BBDR2012}.
\end{proof}
\begin{Proposition}[Korn's inequality in Orlicz spaces]
  \label{thm:korn_orlicz}
  Let $\psi$ be an N-function with $\Delta_2(\psi)$, $
  \Delta_2(\psi^*)<\infty$ and let $G \subset \setR^n$ be a bounded
  John domain.  Then, for all $\bfw \in W_0^{1,\psi}(G)$ 
  \begin{align*}
    \int_{G} \psi\big(\bigabs{\nabla \bfw } \big)\, d\bx &\leq c\, \int_{G}
    \psi\big(\abs{\bfD \bfw}\big)\, d\bx\,.
  \end{align*}
  The constant $c$ depends only on the John constant, $\Delta_2(\psi)$, and
  $\Delta_2(\psi^*)$.
\end{Proposition}
\begin{proof}
  This is a special case of~\cite[Thm.~6.10
]{DieRS10}.
\end{proof}
%
\begin{Proposition}[\Poincare{} inequality]
  \label{lem:Poincare-inequality}
  Let $G \subset \setR^n$ be open and bounded. Let $\psi$ be an
  N-function with $\Delta_2(\psi),\Delta_2(\psi^*)< \infty$.  Then, 
  there exists $c>0$ only depending on $\Delta_2(\psi)$ and
  $\Delta_2(\psi^*)$ such that
  \begin{equation}
    \label{eq:Poincare_2}
    \int_G\psi\bigg( \frac{\abs{u}}{\diameter(G)}\bigg)\, d\bx\leq
    c\int_G\psi(|\nabla u|)\, d\bx\qquad \forall\,u\in
    W^{1,\psi}_0(G). 
  \end{equation}
\end{Proposition}
\begin{proof}
  This is Lemma~6.3 in~\cite{BBDR2012} and is based on the
  properties of the maximal function. 
\end{proof}
\begin{lemma}
  \label{lem:ism}
  Let $G \subset \setR^n$ be a bounded John domain and let $\psi$ be
  an N-function with $\Delta_2(\psi), \Delta_2(\psi^*)<\infty$. Then, 
  for all $q \in L^{\psi^*}_0(G)$ we have
  \begin{align*}
    \norm{q}_{L^{\psi^*}_0(G) } &\le c\, \sup_{\norm{\bv}_{
        W^{1,\psi}_0(G)} \le 1} \skp{q}{\divo \bv},
    \\
     \intertext{and also} 
    \int_G \psi^*(\abs{q})\, d\bx &\le \sup_{\bv
      \in W^{1,\psi}_0(G) } \bigg [\int _G q \divo \bv\, d\bx -
    \frac 1c \, \int_G \psi(\abs{\nabla \bv})\, d\bx\bigg ],
  \end{align*}
  where the constants depend only on $\Delta_2(\psi)$,
  $\Delta_2(\psi^*)$, and the John constant of~$G$.  
\end{lemma}
\begin{proof}
  This is Lemma~4.3 in~\cite{BBDR2012} and is based on the
  properties of the divergence operator.
\end{proof}
We conclude this section by the following anisotropic embedding theorem.
\begin{theorem}
  \label{thm:aniso}
  Let $P\in\partial\Omega$, $g\in W^{1,1}(\Omega_P)$ with $\spt
  g\subset\spt\xi_P$. Let all tangential derivatives $\partial_\tau g$
  satisfy $\partial_\tau g\in L^q(\Omega_P)$, $q>1$, and let also
  $\partial_3 g\in L^r(\Omega_P)$, $r>1$, with
  $\frac{2}{q}+\frac{1}{r}>1$. Then $g\in L^s(\Omega_P)$ with $s$
  given by
  \begin{equation*}
    \frac1s=\frac13 \bigg(\frac{2}q+\frac1r-1\bigg),
  \end{equation*}
    and the following inequalities hold true
  \begin{equation}
    \label{eq:troisi}
    \begin{aligned}
      &      \|g\|_{s,\Omega} \le c\, 
      \|\partial
      _{\tau_1}g\|_{q,\Omega_{P}}^{1/3}  \|\partial _{\tau_2}g\|_{q,\Omega_{P}}^{1/3}\|\partial _3
      g\|_{r,\Omega_{P}}^{1/3}\,,
      \\
      &    \|g\|_{s,\Omega} \le c\, \big (
      \|\partial
      _{\tau_1}g\|_{q,\Omega_{P}} +\|\partial _{\tau_2}g\|_{q,\Omega_{P}}+\|\partial _3
      g\|_{r,\Omega_{P}}\big )\,,
    \end{aligned}
  \end{equation}
  where $c$ depends only on $q$ and $r$.
\end{theorem}
\begin{proof}
  The theorem is proved in~\cite{troisi} provided the Cartesian
  product of smooth open intervals (in that case one can assume that
  different partial derivatives belong to appropriate Lebesgue
  spaces).  The case of the non-flat boundary can be converted to the
  previous one by the coordinate transformation
  $\Phi:(x',x_{3})\to(x',x_3+a(x'))$. By defining $G:=g\circ\Phi$ one
  obtains a function with compact support in the (closed) upper
  half-space and if one defines $\widetilde{G}$ by an even extension
  of $G$ with respect to the $x_3$ direction, it turns out that
  $\widetilde{G}$ belongs to $W^{1,1}_0(K)$, where $K\subset\setR^3$
  is a compact set.  We then apply~\cite[Theorem~1.2]{troisi} to
  $\widetilde{G}$, which can be approximated by functions in
  $C^\infty_0(\setR^3)$, to obtain
  \begin{equation*}
    \begin{aligned}
      \|G\|_{s,\setR^3_+}\leq \|\widetilde{G}\|_{s}&\leq c
      \|\partial_1\widetilde{G}\|^{1/3}_{q}\|\partial_2\widetilde{G}\|^{1/3}_{q}
      \|\partial_3\widetilde{G}\|^{1/3}_{r} 
      \\
      &\leq 2c
      \|\partial_1{G}\|^{1/3}_{q,\setR^3_+}\|\partial_2{G}\|^{1/3}_{q,\setR^3_+}
      \|\partial_3{G}\|^{1/3}_{r,\setR^3_+}.
    \end{aligned}
  \end{equation*}
  Since the Jacobian of the transformation $\Phi$ is equal to one,
  using the reverse transformation to $\Phi$ one gets the first
  statement of the theorem by a change of variables. In fact, by the
  definition $G:=g\circ\Phi$, it turns out that
  \begin{equation*}
 \partial_\alpha G=(\partial_\alpha g+\partial_\alpha a \,\partial_3
  g)\circ\Phi=(\partial_\tau
  g)\circ \Phi,
\end{equation*}
which is a tangential derivative of $g$ composed with $\Phi$, and also
$\partial_3 G=(\partial_3 g)\circ\Phi$.  The additive version is then
proved by Young's inequality.
\end{proof}
%
%
\section{Proof of Theorem~\ref{thm:MT1}}
\label{sec:Stokes}
We assume that $\T$ satisfies the assumption of Theorem~\ref{thm:MT1},
i.e.~$\bS =\bS^0 + \bS^1$ with $\bS^0$ having quadratic-structure and
$\bS^1$ having $(p,\delta)$-structure. Moreover, $\ff$ belongs to
$L^2(\Omega)$.

From the properties of $\bS$ and the standard theory of monotone
operators we easily obtain the existence of a unique
$\bu\in
W^{1,2}_{0,\divo}(\O)$, satisfying for all
$\bv\in 
W^{1,2}_{0,\divo}(\O)$
\begin{equation*}
  \intO\T(\bD\bu)\cdot\bD\bv\, d\bx=\intO \ff \cdot \bv\,
  d\bx\,. 
\end{equation*}
Using Proposition~\ref{lem:hammer}, the properties of $\bS$,
\Poincare's and Korn's inequalities as well as Young's inequality, we
obtain that this solution satisfies the a-priori estimate
\begin{align}
  \label{eq:7}
  \kappa_0(2)\intO \abs{\nabla \bu } ^2 \, d\bx+ \kappa_0(p)\intO
  \varphi(\abs{\nabla \bu }) \, d\bx \le c\, \intO \abs{\ff}^{2} \,
  d\bx\,.
\end{align}
Here and in the sequel we denote by $\kappa_i(2)$ and $\kappa_i(p)$,
$i=0,1$, respectively, the constants in Definition~\ref{ass:1} for $2$
and $p$, respectively. Moreover, in the above estimate and in the
sequel all constants can depend on the characteristics of $\bS^0$ and
$\bS^1$, on $\diameter (\Omega)$, $\abs{\Omega}$, on the space
dimension, and on the John constants of $\Omega$.
Finally, the constants can also depend\footnote{This dependence occurs
  in most cases due to a shift change and results in an additive
  constant.} on $\delta_0$ (cf.~Remark~\ref{rem:delta0}). All these
dependencies will not be mentioned explicitly, while the dependence on
other quantities is made explicit.

It is possible to associate to the $\bu$ a unique pressure $\pi\in
L^{2}_0(\O)$ satisfying for all $\bv\in W^{1,2}_0(\O)$
\begin{equation}
  \label{eq:eq-tlak}
  \intO\T(\bD\bu)\cdot\bD\bv-\pi\diver\bv \, d\bx =\intO \ff \cdot
  \bv\, d\bx\,. 
\end{equation}
Using the properties of $\bS$, Lemma~\ref{lem:ism}, $p <2$,
Lemma~\ref{lem:change2}, and Young's inequality we thus obtain the
following estimate
\begin{align}
  \label{eq:est-tlak}
  \intO\abs{\pi }^{2}\, d\bx \le c\,\Big ( 1 + 
  \norm{\ff}^2_2 
\Big )\,.
\end{align}
\subsection{Regularity in tangential directions and in the
  interior} 
\label{sec:tan}
Let us start with the regularity in tangential directions. The
interior regularity follows along the same lines of reasoning, but
with several simplifications.

The main results of this section are summarized in the following
proposition, which ensures local boundary estimates for tangential
derivatives, that depend on $P \in \bou$ only through
$\norm{\xi_P}_{2,\infty}$ and $\norm{a_P}_{C^{2,1}}$.
\begin{Proposition}
\label{prop:tangential1}
Let the assumptions of Theorem~\ref{thm:MT1} be satisfied and let the
local description $a_P$ of the boundary and the localization function
$\xi_P$ satisfy $(b1)$--\,$(b3)$ and $(\ell 1)$
(cf.~Section~\ref{sec:bdr}). Then, there exist functions $M_0\in
C(\setR^{\ge 0}\times \setR^{\ge 0})$, $M_1 \in
C(\setR^{>0}\times\setR^{\ge 0}\times \setR^{\ge 0})$ such that for
every $P\in \partial \Omega$
\begin{equation}
  \label{eq:tang-final}
  \begin{aligned}
    \intO \xi^2_P \bigabs{ \partial_\tau \nabla \bu }^2+& \xi^2_P
    \bigabs{ \nabla\partial_\tau \bu }^2 + \phi (\xi_P
    \abs{\partial_\tau \nabla \bu})+\phi (\xi_P
    \abs{\nabla \partial_\tau \bu})\,d\bx
    \\
    &+ \int_\Omega \xi^2_P \bigabs{\partial _\tau \F (\bD\bu)
    }^2\,d\bx\leq M_0(\norm{\xi_P}_{2,\infty},\norm{a_P}_{C^{2,1}}) 
    \Big(1+\norm{\bff}_2^2
    \Big)\,,
  \end{aligned}
\end{equation}
and 
\begin{align}
  \label{eq:tang-final-pres}
  \intO \xi^2_P \bigabs{ \partial_\tau \pi }^2 \, d\bx \le M_1(\delta^{p-2},
  \norm{\xi_P}_{2,\infty},\norm{a_P}_{C^{1,1}}) \Big ( 1+ \norm{\bff}_2^2
  \Big)\,,
\end{align}
where $\norm{a_P}_{C^{k,1}}$ means
$\norm{a_P}_{C^{k,1}(\overline{B^2_{3/4R_P}(0)})}$, for $k=1,2$.
\end{Proposition}
As usual in the study of boundary regularity we need to localize and
to use appropriate test functions. Consequently, let us fix
$P\in \partial \Omega$
and in $\O_P$ use $\xi:=\xi_P$, $a:=a_P$, 
while $h\in(0,\frac{R_{P}}{16})$,  as in Section~\ref{sec:bdr}. We use as test function
$\bv$ in the weak formulation~\eqref{eq:eq-tlak}
\begin{equation*}
  \bv=\difn{(\xi\,\bpsi)},
\end{equation*}
(more precisely $\xi$ is extended by zero for 
$\bx\in\Omega\backslash\Omega_P$, in order to have a global function
over $\Omega$) with $\bpsi\in W^{1,2}_0(\O)$ to get, with the help of
Lemma~\ref{lem:TD1} and Lemma~\ref{lem:TD3}, the following equality
\begin{equation}
  \label{eq:dtTx1}
  \begin{aligned}
    \intO& \difp{\T(\bD\bu)}\cdot
    \bD(\xi\,\bpsi)+\T(\bD\bu)\cdot\big(\tran{(\partial_3(\xi\,\bpsi))}
    \otimess\difn\nabla a\big) -\pi\divo \difn(\xi\,\bpsi) \, d\bx
    \\
    &\qquad\qquad=\intO \ff\cdot\difn(\xi\, \bpsi)\, d\bx\,.
  \end{aligned}
\end{equation}
Due to the fact that $\bu\in W^{1,2}_{0,\divo}(\O)$ we can set
$\bpsi=\xi\,\difp (\bu_{|\widetilde{\Omega}_P})$ in $\Omega_P$ (and
zero outside), hence as a test function we can consider the following
vector field
\begin{equation*}
  \bv=\difn{(\xi^2\difp(\bu_{|\widetilde{\Omega}_P}))},
\end{equation*}
where $\widetilde{\Omega}_P:=\frac{1}{2}\Omega_P$, for the definition
recall~\eqref{eq:scaled_omega_P}.  Since $\bpsi$ has zero trace on
$\O_P$, we get $\bv\in W^{1,2}_0(\Omega_P)$, for small enough $h>0$.
\begin{remark}
  As a general disclaimer, we stress that in the sequel we will use
  the convention that functions are extended by zero off their proper
  set of definition, when needed. We avoid making this explicit in the
  sequel to avoid cumbersome expressions.
\end{remark}
Using Lemma~\ref{lem:TD1}--Lemma~\ref{lem:TD3} we thus get the
following identity
\begin{equation}
  \label{eq:dtTx2}
  \begin{aligned}
    \intO& \xi^2\difp{\T(\bD\bu)}\cdot \difp \bD\bu \, d\bx
    \\
    =&-\intO \T(\bD\bu)\cdot\big(\xi ^2 \difp \partial _3 \bu
    -(\xi_{-\tau } \difn\xi +\xi \difn\xi) \partial_3\bu \big)
    \otimess\difn\nabla a\, d\bx
    \\
    &-\intO \T(\bD\bu)\cdot \xi^2 \trap{(\partial _3\bu)}\otimess
    \difn{\difp{\nabla a}} - \T(\bD\bu)\cdot\difn{\big(2\xi \nabla \xi
      \otimess \difp\bu\big)}\, d\bx
    \\
    &+\intO \T(\trap{(\bD\bu)})\cdot \big(2 \xi \partial_3\xi \difp\bu +
    \xi^2 \difp\partial_3\bu \big)\otimess\difp\nabla a \, d\bx
    \\
    &-\intO \pi \big (\xi^2 \difn{ \difp{\nabla a}} -(\tran{\xi}\difn
    {\xi} + \xi \difn{\xi} )\difn{\nabla a}\big ) \cdot \partial _3\bu
    \, d\bx
    \\
    &-\intO \pi\big ( \difn{(2\xi \nabla \xi\cdot\difp{\bu})} -\xi^2
    \difp{\partial_3} \bu \cdot \difp{\nabla a} \big) \, d\bx
    \\
    &+\intO \trap{\pi}\big ( 2\xi \partial_3 \xi \difp{\bu} +\xi^2
    \difp{\partial_3} \bu \big )\cdot \difp{\nabla a}  \, d\bx
    \\
    &+\intO\ff\cdot\difn(\xi^2 \difp \bu)\, d\bx=:\sum_{j=1}^{15} I_j\,.
  \end{aligned}
\end{equation}
The term providing the information concerning the regularity of the
solution is the integral kept on the left-hand side. From the
assumption on $\bS$ and Proposition~\ref{lem:hammer} it may be
estimated from below by
\begin{align}
  \begin{aligned}
    \label{eq:odhadT2}
    \intO\xi ^2 \difp{\T(\bD\bu)}\cdot \difp \bD\bu \, d\bx &\geq
    c\,\intO \xi^2 \bigabs{ \difp \bD\bu }^2 +\xi^2 \bigabs{ \difp \F
      (\bD\bu) }^2\, d\bx \,. 
  \end{aligned}
\end{align}
%
We take advantage of the restriction $p\le 2$, especially to 
gain further information from the right-hand side
of~\eqref{eq:odhadT2}, as in the following lemma.
\begin{lemma}
  \label{lem:p-est}
  Let $\bF$ be given by~\eqref{eq:def_F} for some $p \in (1,2]$ and
  $\delta\ge 0$. Then, for $\xi , a$ as above and
  $\bu \in W^{1,p}_{0}(\Omega)$ we have\footnote{Note, that the
    constants here do not depend on $\delta_0$.}
  \begin{equation*}
    \begin{aligned}
      \intO \phi\big(\xi \abs{\nabla \difp\bu}\big) + \phi\big(\xi
      \abs{ \difp \nabla \bu}\big)\, d\bx &\le c\,\intO \xi^2 \bigabs{
        \difp \F (\bD\bu) }^2 \, d\bx
      \\
      &\quad + c( \norm{\xi}_{1,\infty},
      \norm{a}_{C^{1,1}})\hspace*{-4mm} \int_{\Omega\cap \spt
        \xi}\hspace*{-4mm} \phi \big (\abs{ \nabla\bu }\big ) \,
      d\bx\,. \hspace*{-3mm}
    \end{aligned}
  \end{equation*}
\end{lemma}
\begin{proof}
  Using Lemma~\ref{lem:TD1} and the convexity of $\phi$ we see
  \begin{equation}
    \label{eq:p-est1}
    \begin{aligned}
      \intO &\phi (\xi \abs{\difp\nabla \bu}) + \phi (\xi \abs{\nabla
        \difp \bu})\, d\bx
      \\
      &\le c\,\intO \phi (\xi \abs{\nabla \difp \bu}) + \phi (\xi\abs{
        \trap{(\partial_3 \bu)} \difp \nabla a })\, d\bx
      \\
      &\le c\,\intO \phi (\xi \abs{\nabla \difp \bu}) \, d\bx + c(
      \norm{a}_{C^{1,1}}) \hspace*{-4mm} \int_{\Omega\cap \spt
        \xi}\hspace*{-4mm} \phi (\abs{\nabla \bu})\, d\bx\,,
    \end{aligned}
  \end{equation}
  where we also used a change of variables in the estimate of the last
  term. To treat the first term from the right-hand side we use the
  following identity
  \begin{equation*}
    \xi\, \nabla \difp \bu = \nabla (\xi\, \difp \bu )
    -\nabla\xi\otimes \difp \bu,
  \end{equation*}
 and consequently we get
 \begin{equation*}
   \begin{aligned}
     \intO \phi (\xi \abs{\nabla \difp \bu}) \, d\bx &\le c \intO \phi
     (\abs{\bD (\xi \,\difp \bu)}) \, d\bx + c(\norm{\xi}_{1,\infty})
     \hspace*{-4mm} \int_{\Omega\cap \spt
       \xi}\hspace*{-4mm} \phi (\abs{\nabla \bu}) \, d\bx \,,
   \end{aligned}
 \end{equation*}
 where we used Korn's inequality
 (cf. Proposition~\ref{thm:korn_orlicz}) and
 \begin{equation}
   \label{eq:diff}
   \int_{\Omega\cap \spt \xi} \hspace*{-4mm}\phi(\abs{\difpm\bu}) \, d\bx \le 
   \int_{\Omega\cap \spt \xi} \hspace*{-4mm}\phi(\abs{\nabla\bu}) \, d\bx \,, 
 \end{equation}
 which can be proved in the same way as in the classical $L^p$-setting
 (cf.~\cite{evans-pde}). Using the identities
\begin{align*}
  \bD(\xi\, \difp\bu) &= \xi\, \bD(\difp \bu) + \nabla \xi \otimess
  \difp \bu
  \\
  & = \xi\, \difp \bD\bu + \trap{(\partial _3 \bu)} \otimess
  \difp\nabla a + \nabla \xi \otimess \difp \bu,
\end{align*}
the properties of $\varphi, \xi, a$ and~\eqref{eq:diff}, we obtain
\begin{equation}
  \label{eq:p-est2}
  \begin{aligned}
    \intO \hspace*{-1mm} \phi (\xi \abs{\nabla \difp \bu}) \, d\bx
    &\le c\hspace*{-1mm}\intO \hspace*{-1mm}\phi (\xi \abs{\difp
      \bD\bu}) \, d\bx
    + c(\norm{\xi}_{1,\infty}, \norm{a}_{C^{1,1}})
    \hspace*{-4mm} \int_{\Omega\cap \spt
      \xi}\hspace*{-4mm} \phi (\abs{\nabla \bu}) \, d\bx \,.
  \end{aligned}
\end{equation}
Using Lemma~\ref{lem:change2},~\eqref{eq:equi1},~\eqref{eq:equi2},
$p\le 2$ (or more generally that $\phi''$ is non-increasing), and
Proposition~\ref{lem:hammer}, we also obtain
\begin{align*}
  \phi (\xi \abs{\difp \bD\bu}) &\le c\big (\phi_{\abs{\bD\bu}
    +\abs{d^+ \bD\bu}} (\xi \abs{\difp \bD\bu}) +\phi (\abs{\bD\bu} +
  \abs{d^+ \bD\bu}) \big )
  \\
  &\sim \big( \delta +\abs{\bD\bu} +\abs{d^+ \bD\bu} +\xi \abs{ \difp
    \bD\bu}\big)^{p-2} \abs{\xi \difp \bD\bu}^2 +\phi (\abs{\bD\bu} +
  \abs{d^+ \bD\bu})
   \\
   &\le c\, \big( \delta +\abs{\bD\bu} +\abs{d^+ \bD\bu}\big)^{p-2}
   \abs{\xi \difp \bD\bu}^2 +c\,\phi (\abs{\bD\bu} + \abs{d^+ \bD\bu})
   \\
   &\sim \bigabs{\difp \bF(\bD\bu )}^2 +\phi (\abs{\bD\bu} + \abs{d^+
     \bD\bu})\,.
 \end{align*}
 Inserting this into~\eqref{eq:p-est2} we get
 \begin{equation*}
   \begin{aligned}
     \intO \phi (\xi \abs{ \nabla \difp \bu}) \, d\bx &\le c \intO
     \bigabs{\difp \bF(\bD\bu )}^2 \, d\bx
     + c(\norm{\xi}_{1,\infty}, \norm{a}_{C^{1,1}})\hspace*{-4mm} \int_{\Omega\cap \spt
       \xi}\hspace*{-4mm} \phi (\abs{\nabla \bu}) \, d\bx \,,
   \end{aligned}
 \end{equation*}
 which together with~\eqref{eq:p-est1} yields the assertion.
\end{proof}
%


We can now give the proof of the regularity of tangential derivatives.
\begin{proof}[Proof of Proposition~\ref{prop:tangential1}]
  Using the previous lemma for the function $\phi$ with exponents $2$
  and $p$ and observing that for any smooth enough function $f$
  \begin{equation}
    \label{eq:two-derivatives-a}
    \nabla\partial_\tau  f=\partial_\tau
    \nabla f+\nabla\partial_{\alpha} a \,\partial_3  f, 
  \end{equation}
  (valid if $a\in C^{2,1}$), we derive
  from~\eqref{eq:dtTx2} and~\eqref{eq:odhadT2} the following estimate
  \begin{align}
    \intO &\xi^2 \bigabs{ \difp \nabla \bu }^2 +\xi^2\bigabs{\nabla
      \difp\bu }^2 +\xi^2 \bigabs{ \difp \F (\bD\bu) }^2 + \phi (\xi
    \abs{\difp\nabla \bu}) + \phi (\xi \abs{\nabla \difp \bu})\, d\bx
    \notag
    \\
    &\qquad \le c \, \sum_{j=1}^{15} I_j+
    c(\norm{\xi}_{1,\infty},\norm{a}_{C^{1,1}}) \hspace*{-3mm}
    \int_{\Omega\cap \spt \xi}\hspace*{-3mm} \abs{\nabla \bu }^2 +\phi
    \big (\abs{ \nabla\bu }\big ) \, d\bx\,.  \label{eq:tang}
  \end{align}
  It is very relevant that we have a full gradient $\nabla\bu$,
  instead of $\bD\bu$ in all terms on the left-hand side, except that
  involving $\bF(\bD\bu)$, at the price of a lower order term in the
  right-hand side, thanks to Lemma~\ref{lem:p-est}.  We are now going
  to estimate the terms $I_j$, for $j=1,\ldots,15$. By using the
  growth properties of $\bS^0$ and $\bS^1$,
  Proposition~\ref{lem:hammer}, and Young's
  inequality~\eqref{eq:ineq:young} we obtain for the terms with $\bS$
  the following estimates, which are valid for any given
  $\varepsilon>0$:
\begin{align}
  \bigabs{I_1} &\le c\, (\norm{a}_{C^{1,1}}) \intO \xi^2 \big (\abs
  {\bD \bu} + \phi'(\abs{\bD\bu})\big ) \abs {\difp \partial _3 \bu
  }\, d\bx \label{eq:i1}
  \\
  &\le c\, (\varepsilon^{-1},\norm{a}_{C^{1,1}}) \hspace*{-3mm}
  \int_{\Omega\cap \spt \xi}\hspace*{-3mm}\abs{\bD\bu}^2 +
  \phi(\abs{\bD\bu}) \, d\bx + \varepsilon \intO \xi^2 \bigabs{ \difp
    \nabla \bu }^2 + \phi (\xi\abs{\difp\nabla \bu}) \, d\bx \,,
    \notag 
\end{align}
\begin{equation}
   \label{eq:i23}
  \begin{aligned}
    \bigabs{I_2 + I_3} &\le c\, (\norm{\xi}_{1,\infty},\norm{a}_{C^{1,1}})
 \hspace*{-3mm} \int_{\Omega\cap \spt
      \xi}\hspace*{-3mm} \big (\abs {\bD \bu} +
    \phi'(\abs{\bD\bu})\big ) \abs {\partial _3 \bu }\,
    d\bx 
    \\
    &\le c\, (\norm{\xi}_{1,\infty},\norm{a}_{C^{1,1}})
    \hspace*{-3mm} \int_{\Omega\cap \spt
      \xi}\hspace*{-3mm}\abs{\nabla \bu}^2 + \phi(\abs{\nabla\bu}) \,
    d\bx\,. 
  \end{aligned}
\end{equation}
The next term requires the full regularity of $a\in C^{2,1}$, in fact
\begin{equation}
  \label{eq:i4}
  \begin{aligned}
    \bigabs{I_4} &\le c\, (\norm{a}_{C^{2,1}}) 
    \hspace*{-3mm} \int_{\Omega\cap \spt
      \xi}\hspace*{-3mm} \big (\abs {\bD \bu} +
    \phi'(\abs{\bD\bu})\big ) \abs {\trap{(\partial _3 \bu)} }\,
    d\bx 
    \\
    &\le c\, (\norm{a}_{C^{2,1}})
    \hspace*{-3mm} \int_{\Omega\cap \spt
      \xi}\hspace*{-3mm}\abs{\nabla \bu}^2 + \phi(\abs{\nabla\bu}) \,
    d\bx\,, 
  \end{aligned}
\end{equation}
where we also used a translation argument for $\partial _3 \bu$.
Also using~\eqref{eq:diff},~\eqref{eq:est-l} we get 
\begin{equation}
  \label{eq:i5}
  \begin{aligned}
    \bigabs{I_5} &\le c(\epsilon ^{-1})\hspace*{-3mm} \int_{\Omega\cap
      \spt \xi}\hspace*{-3mm} \abs {\bD \bu}^2+ \phi(\abs {\bD
      \bu}) \, d\bx
    \\
    &\qquad+ \frac \varepsilon 4 \int_{\Omega\cap \spt
      \xi}\hspace*{-3mm} \bigabs{\difn{\big(2\xi \nabla \xi \otimess
        \difp\bu\big)}}^2 +\phi\big (\abs{\difn{\big(2\xi \nabla \xi
        \otimess \difp\bu\big)}}\big )\, d\bx
    \\
    &\le c(\epsilon ^{-1})\hspace*{-3mm} \int_{\Omega\cap \spt
      \xi}\hspace*{-3mm}\abs{\bD \bu}^2 + \phi(\abs{\bD\bu}) \, d\bx
    \\
    &\qquad + \frac \varepsilon 4  \int_{\Omega\cap \spt
      \xi}\hspace*{-3mm} \bigabs{\nabla{\big(2\xi \nabla \xi
        \otimess \difp\bu\big)}}^2 +\phi\big (\abs{\nabla{\big(2\xi
        \nabla \xi \otimess \difp\bu\big)}}\big )\, d\bx
    \\
    &\le c(\epsilon ^{-1},\norm{\xi}_{2,\infty}) \hspace*{-3mm}
    \int_{\Omega\cap \spt \xi}\hspace*{-3mm}\abs{\nabla \bu}^2 +
    \phi(\abs{\nabla\bu}) \, d\bx
    \\
    &\qquad + \varepsilon \,\max \big(1,4\norm{\xi}_{1,\infty}^2\big )
    \hspace*{-3mm} \int_{\Omega\cap \spt
      \xi}\hspace*{-3mm}\xi^2\abs{\nabla \difp\bu}^2 +
    \phi(\abs{\xi\nabla\difp\bu}) \, d\bx\,,
  \end{aligned}    
\end{equation}
and observe that, by the definition of $\xi=\xi_P$ we have
$\max(1,4\norm{\xi}_{1,\infty}^2)=4\norm{\xi}_{1,\infty}^2$.  
Using also a translation argument for $\bD \bu$
and~\eqref{eq:diff} yields
\begin{equation}
   \label{eq:i6}
  \begin{aligned}
    \bigabs{I_6} &\le c\, (\norm{\xi }_{1,\infty},\norm{a}_{C^{1,1}})
    \hspace*{-3mm} \int_{\Omega\cap \spt
      \xi}\hspace*{-3mm} \big (\abs {\trap{(\bD \bu)}}^2 +
    \phi'(\abs{\trap{(\bD\bu)}})\big ) \abs {\difp{ \bu} }\, d\bx
    \\
    &\le c\, (\norm{\xi }_{1,\infty},\norm{a}_{C^{1,1}})
    \hspace*{-3mm} \int_{\Omega\cap \spt
      \xi}\hspace*{-3mm}\abs{\nabla \bu}^2 + \phi(\abs{\nabla\bu}) \,
    d\bx\,.
  \end{aligned}
\end{equation}
Passing to $I_7$ we have
\begin{align}
  \bigabs{I_7}
  &\le c\, (\norm{a}_{C^{1,1}}) \intO \xi^2 \big (\abs
    {\trap{(\bD \bu)}} + \phi'(\abs{\trap{(\bD\bu)}})\big ) \abs
    {\difp{\partial _3 \bu} }\, d\bx \label{eq:i7}
  \\
  &\le c\, (\varepsilon^{-1},\norm{a}_{C^{1,1}}) \hspace*{-3mm}
    \int_{\Omega\cap \spt \xi}\hspace*{-3mm}\abs{\bD\bu}^2 +
    \phi(\abs{\bD\bu}) \, d\bx + \varepsilon \intO \xi^2 \bigabs{ \difp
    \nabla \bu }^2 + \phi (\xi\abs{\difp\nabla \bu}) \,
    d\bx \,, \notag 
\end{align}
where we again used a translation argument for $\bD \bu$. The terms
with the pressure are estimated, using Young's inequality, as follows
\begin{equation}
   \label{eq:i8}
  \begin{aligned}
    \bigabs{I_8} &\le c\, (\norm{a}_{C^{2,1}})
    \hspace*{-3mm} \int_{\Omega\cap \spt \xi}
    \hspace*{-3mm} \abs
    {\pi}\abs {\partial _3 \bu }\, d\bx \le c\, (\norm{a}_{C^{2,1}})
    \hspace*{-3mm} 
    \int_{\Omega\cap \spt
      \xi}\hspace*{-3mm} \abs{\pi}^2 + \abs{\nabla \bu}^2 \, d\bx\,,
  \end{aligned}
\end{equation}
\begin{equation}
  \label{eq:i910}
  \begin{aligned}
    \bigabs{I_9 +I_{10}} &\le c\, (\norm{a}_{C^{1,1}},\norm{\xi}_{1,\infty}) \hspace*{-3mm}
    \int_{\Omega\cap \spt \xi}\hspace*{-3mm} \abs {\pi}\abs {\partial
      _3 \bu }\, d\bx
    \\
    &\le c\, (\norm{a}_{C^{1,1}},\norm{\xi}_{1,\infty}) 
    \hspace*{-3mm} \int_{\Omega\cap \spt \xi}\hspace*{-3mm}
    \abs{\pi}^2 + \abs{\nabla \bu}^2 \, d\bx\,,
  \end{aligned}
\end{equation}
\begin{equation}
  \label{eq:i11}
  \begin{aligned}
    \bigabs{I_{11}}    &\le c(\epsilon ^{-1})\hspace*{-3mm}
    \int_{\Omega\cap \spt \xi}\hspace*{-3mm} \abs {\pi}^2 \, d\bx +
    \varepsilon \hspace*{-3mm} \int_{\Omega\cap \spt
      \xi}\hspace*{-3mm} \bigabs{\nabla{\big(2\xi \nabla \xi \cdot 
        \difp\bu\big)}}^2 \, d\bx
    \\
    &\le c(\epsilon ^{-1},\norm{\xi}_{2,\infty}) \hspace*{-3mm}
    \int_{\Omega\cap \spt \xi}\hspace*{-3mm}\abs{\pi}^2+
    \abs{\nabla \bu}^2\, d\bx
    + 4\varepsilon \,\norm{\xi}_{1,\infty}^2
    \hspace*{-3mm} \int_{\Omega\cap \spt
      \xi}\hspace*{-3mm}\xi^2\abs{\nabla \difp\bu}^2\, d\bx\,,
  \end{aligned}    
\end{equation}
where we also used~\eqref{eq:diff},
\begin{equation}
  \label{eq:i12}
  \begin{aligned}
    \bigabs{I_{12}} &\le c\, (\norm{a}_{C^{1,1}})
    \hspace*{-3mm} \int_{\Omega\cap \spt \xi}\hspace*{-3mm} \xi^2\abs
    {\pi}\abs {\difp\partial _3 \bu }\, d\bx
    \\
    &\le c\, (\epsilon^{-1},\norm{a}_{C^{1,1}}) \hspace*{-3mm}
    \int_{\Omega\cap \spt \xi}\hspace*{-3mm} \abs{\pi}^2\, d\bx +
    \epsilon \int_{\Omega} \xi^2\abs{\difp\nabla \bu}^2 \, d\bx\,,
  \end{aligned}
\end{equation}
\begin{equation}
  \label{eq:i13}
  \begin{aligned}
    \bigabs{I_{13}} &\le c\, (\norm{\xi}_{1,\infty},\norm{a}_{C^{1,1}})
    \hspace*{-3mm} \int_{\Omega\cap \spt \xi}\hspace*{-3mm} \abs
    {\trap\pi}\abs {\difp \bu }\, d\bx
    \\
    &\le c\, (\norm{\xi}_{1,\infty},\norm{a}_{C^{1,1}}) \hspace*{-3mm}
    \int_{\Omega\cap \spt \xi}\hspace*{-3mm} \abs{\pi}^2+ \abs{\nabla
      \bu}^2 \, d\bx\,,
  \end{aligned}
\end{equation}
where we also used a translation argument for $\pi$ and~\eqref{eq:diff},
\begin{equation}
  \label{eq:i14}
  \begin{aligned}
    \bigabs{I_{14}} &\le c\, (\norm{a}_{C^{1,1}})
    \hspace*{-3mm} \int_{\Omega\cap \spt \xi}
    \hspace*{-3mm} \xi^2\abs
    {\trap \pi}\abs {\difp\partial _3 \bu }\, d\bx
    \\
    &\le c\, (\vep^{-1},\norm{a}_{C^{1,1}}) \hspace*{-3mm}
    \int_{\Omega\cap \spt \xi}\hspace*{-3mm} \abs{\pi}^2\, d\bx +
    \epsilon \int_{\Omega} \xi^2\abs{\difp\nabla \bu}^2 \, d\bx\,,
  \end{aligned}
\end{equation}
where we again used a translation argument for $\pi$
and~\eqref{eq:diff}. The term with the external force is estimated as
follows
\begin{equation}
  \label{eq:i15}
  \begin{aligned}
    \bigabs{I_{15}} &\le c(\epsilon ^{-1})\hspace*{-3mm}
    \int_{\Omega\cap \spt \xi}\hspace*{-3mm} \abs {\ff}^2 \, d\bx +
    \varepsilon \hspace*{-3mm} \int_{\Omega\cap \spt
      \xi}\hspace*{-3mm} \bigabs{\difn{\big(\xi^2 \difp\bu\big)}}^2 \,
    d\bx
    \\
    &\le c(\epsilon ^{-1}) \hspace*{-3mm} \int_{\Omega\cap \spt
      \xi}\hspace*{-3mm} \abs {\ff}^2 \, d\bx + \varepsilon
    \hspace*{-3mm} \int_{\Omega\cap \spt \xi}\hspace*{-3mm}
    \bigabs{\nabla{\big(\xi^2 \difp\bu\big)}}^2 \, d\bx
    \\
    &\le c(\epsilon ^{-1},\norm{\xi}_{1,\infty}) \hspace*{-3mm}
    \int_{\Omega\cap \spt \xi}\hspace*{-3mm}\abs{\ff}^2+
    \abs{\nabla \bu}^2\, d\bx + \varepsilon  \hspace*{-3mm} \int_{\Omega\cap \spt
      \xi}\hspace*{-3mm}\xi^2 \abs{\nabla \difp\bu}^2\, d\bx
  \end{aligned}    
\end{equation}
where we also used~\eqref{eq:diff}.

Choosing in the estimates~\eqref{eq:i1}--\eqref{eq:i15} the constant
$\epsilon>0$
\begin{equation*}
  \epsilon=\frac{1}{56\norm{\xi}_{1,\infty}^2},
\end{equation*}
we can absorb all terms involving tangential increments of
$\nabla\bu$ in the left-hand side of~\eqref{eq:tang}, obtaining the
following fundamental estimate
\begin{align*}
  \intO &\xi^2 \bigabs{ \difp \nabla \bu }^2 +\xi^2\bigabs{\nabla
    \difp\bu }^2 +\xi^2 \bigabs{ \difp \F (\bD\bu) }^2 + \phi (\xi
  \abs{\difp\nabla \bu}) + \phi (\xi \abs{\nabla \difp \bu})\, d\bx
  \notag 
  \\
  &\le c( \norm{\xi}_{2,\infty},\norm{a}_{C^{2,1}}) \hspace*{-3mm}
  \int_{\Omega\cap \spt \xi}\hspace*{-3mm} \abs{\ff}^2 +\abs{\pi}^2 +
  \abs{\nabla \bu}^2 + \phi(\abs{\nabla \bu}) \, d\bx\,,   
\end{align*}
where the right-hand side is finite (and independent of $h$)
since~\eqref{eq:7} and~\eqref{eq:est-tlak} imply
\begin{align}
  \intO &\xi^2 \bigabs{ \difp \nabla \bu }^2 +\xi^2\bigabs{\nabla
          \difp\bu }^2 +\xi^2 \bigabs{ \difp \F (\bD\bu) }^2 + \phi (\xi
          \abs{\difp\nabla \bu}) + \phi (\xi \abs{\nabla \difp \bu})\, d\bx
          \notag
  \\
        &\le M_0( \norm{\xi}_{2,\infty},\norm{a}_{C^{2,1}}) 
          \big (1+\norm{\ff}^2 \big ).   \label{eq:tang1}
\end{align}
From estimate~\eqref{eq:tang1} and the properties of
tangential derivatives recalled in~\eqref{eq:2a}, we thus
obtain~\eqref{eq:tang-final}.
%
%

To prove estimate~\eqref{eq:tang-final-pres} for $\partial _\tau \pi$ 
we start with
\begin{align}
  \label{eq:tangpres1} 
  \intO \xi ^2 \abs{\difp \pi }^2 \, d\bx &\le 2 \intO \abs{\xi\,
    \difp \pi -\mean{\xi \difp \pi}_\O }^2 \, d\bx+ \frac
  2{\abs{\Omega}} \,\Bigabs{\intO {\xi\, \difp \pi } \, d\bx}^2,
\end{align}
to take advantage of the \Poincare\ inequality.  The second term on
the right-hand side is treated as follows
\begin{equation}
  \begin{aligned}
    \label{eq:tangpres2}
    \frac 2{\abs{\Omega}} \,\Bigabs{\intO {\xi \,\difp \pi } \,
      d\bx}^2&= \frac 2{\abs{\Omega}} \,\Bigabs{\int_{\Omega\cap \spt
        \xi}\hspace*{-3mm} \pi\, \difn \xi \, d\bx}^2
\le 2\,\norm{\xi}_{1,\infty} ^2 \int_{\Omega\cap \spt
      \xi}\hspace*{-3mm} \abs{\pi}^2\, d\bx\,,
  \end{aligned}
\end{equation}
where we used Lemma~\ref{lem:TD3}. The first term on the right-hand
side of~\eqref{eq:tangpres1} is treated with the help of
Lemma~\ref{lem:ism}. For that we re-write~\eqref{eq:dtTx1}, using
Lemma~\ref{lem:TD1} and Lemma~\ref{lem:TD3}, and get for all $\bpsi\in 
W^{1,2}_0(\O)$
\begin{equation*}
\begin{aligned}
  &\intO \xi\, \difp \pi \divo \bpsi \, d\bx 
  \\
  &=\intO \xi\, \difp{\T(\bD\bu)}\cdot \bD\bpsi + \bS(\bD \bu) \cdot
  \difn{(\nabla \xi \otimess \bpsi)} -
  \T(\trap{(\bD\bu)})\cdot\big({\partial_3(\xi\bpsi)}\otimess\difp\nabla
  a\big) \, d\bx \notag
  \\
  &\quad + \intO \trap \pi\, {\partial_3(\xi\bpsi)}\cdot\difp\nabla a
  -\pi \, \difn{(\nabla \xi \cdot \bpsi)}\, d\bx -\intO
  \ff\cdot\difn(\xi \bpsi)\, d\bx=: \sum_{k=1}^6 J_k\,.\notag 
\end{aligned}
\end{equation*}
Thus, we have 
\begin{equation}
  \label{eq:tangpres3}
  \begin{aligned}
    \intO &\abs{\xi\, \difp \pi -\mean{\xi\, \difp \pi}_\O }^2 \, d\bx
    \\
    &\le \sup_{\bpsi \in W^{1,2}_0(\O)} \bigg [\intO \big(\xi \difp
    \pi -\mean{\xi \difp \pi}_\O\big)\divo \bpsi\, d\bx - \frac 1{C_0}
    \, \intO \abs{\nabla \bpsi}^2\, d\bx\bigg ]
    \\
    &= \sup_{\bpsi \in W^{1,2}_0(\O)} \bigg [\sum_{k=1}^6 J_k -\intO
    \mean{\xi \difp \pi}_\O\,\divo \bpsi\, d\bx - \frac 1{C_0} \, \intO
    \abs{\nabla \bpsi}^2\, d\bx\bigg ]\,,
  \end{aligned}
\end{equation}
where $C_0$ does only depend on the John constants of $\Omega$.  Using
Young's inequality, the fact that the measure of $\Omega$ is finite,
and~\eqref{eq:tangpres2} we have, for all $\epsilon>0$,
\begin{equation*}
  \begin{aligned}
    \Bigabs{\intO \mean{\xi\, \difp \pi}_\O\,\divo \bpsi\, d\bx} &\le
    \epsilon \, \intO \abs{\nabla \bpsi}^2 \, d\bx+ c(\epsilon ^{-1})
    \, \norm{\xi}_{1,\infty} ^2 \int_{\Omega\cap \spt
      \xi}\hspace*{-3mm} \abs{\pi}^2\, d\bx\,.
  \end{aligned}
\end{equation*}
The terms $J_k$, $k=1,\ldots,6$, are estimated, using the properties
of $\bS^0$, $\bS^1$, and Young's inequality, as follows. We denote by
$C_{P}$ the global \Poincare{} constant for $W^{1,2}_{0}(\Omega)$,
depending only on $\abs {\Omega}$. We have first
\begin{equation}
  \label{eq:j1}
  \begin{aligned}
    \bigabs{J_1} &\le c\, \intO \xi \big ( \abs {\difp \bD\bu}
  +\delta^{\frac {p-2}2}\,\abs{\difp\bF (\bD\bu)}\big )\abs{\nabla
    \bpsi} \, d\bx
  \\
  &\le \epsilon \, \intO \abs{\nabla \bpsi}^2 \, d\bx + c(\epsilon
  ^{-1}) \hspace*{-3mm} \int_{\Omega\cap \spt \xi}\hspace*{-3mm} \xi^2
  \abs {\difp \bD\bu}^2 + \delta^{ {p-2}}\,\xi^2 \, \abs{\difp \bF
    (\bD\bu)}^2 \, d\bx\,, 
\end{aligned}
\end{equation}
where we also used the equivalence~\eqref{eq:3} and $p\le 2$. Next, we estimate 
\begin{align}
  \bigabs{J_2} &\le \epsilon \, \intO \abs{\difn (\nabla \xi \otimess
    \bpsi)}^2 \, d\bx + c(\epsilon ^{-1}) \hspace*{-3mm}
  \int_{\Omega\cap \spt \xi}\hspace*{-3mm} \abs {\bD\bu}^2 + \delta^{
    {p-2}}\,\phi(\abs{\bD\bu}) \, d\bx \notag
  \\
  &\le \epsilon \, \intO \abs{\nabla (\nabla \xi \otimess \bpsi)}^2 \,
  d\bx + c(\epsilon ^{-1}) \hspace*{-3mm} \int_{\Omega\cap \spt
    \xi}\hspace*{-3mm} \abs {\bD\bu}^2 + \delta^{
    {p-2}}\,\phi(\abs{\bD\bu}) \, d\bx \label{eq:j2}
  \\
  &\le \epsilon \,2 (1+C_{P}) \norm{\xi}_{2,\infty} ^2
 \intO \abs{\nabla \bpsi}^2 \,
  d\bx+ c(\epsilon ^{-1}) \hspace*{-3mm} \int_{\Omega\cap \spt
    \xi}\hspace*{-3mm} \abs {\bD\bu}^2 + \delta^{
    {p-2}}\,\phi(\abs{\bD\bu}) \, d\bx\,, \notag  
\end{align}
%
where we also used $p\le 2$,~\eqref{eq:diff}, and \Poincare's
inequality. Next, 
\begin{align}
  \hspace*{-1mm}\bigabs{J_3} 
  &\le c\,\norm{a}_{C^{1,1}}\,\intO \big ( \abs {\trap {(
    \bD\bu)}} +\delta^{\frac {p-2}2}\big (\varphi(\abs{\trap
    {(\bD\bu)}})\big )^{\frac 12}\big ) \big (\abs{(\nabla \xi )\bpsi}
    + \abs{\xi \nabla \bpsi}\big ) \, d\bx   \label{eq:j3}
  \\
  &\le \epsilon \,2 (1+C_{P}) \norm{\xi}_{1,\infty}^2\intO
    \abs{\nabla \bpsi}^2 \, d\bx + c(\epsilon ^{-1})
    \,\norm{a}^2_{C^{1,1}} \hspace*{-4mm}
    \int_{\Omega\cap \spt \xi}\hspace*{-3mm} \abs {\bD\bu}^2 +
    \delta^{{p-2}}\,\phi(\abs{\bD\bu}) \, d\bx\,, \notag
\end{align}
%
where we also used a translation argument for $\bD\bu$ and \Poincare's
inequality. Next, we obtain 
\begin{equation}
  \label{eq:j4}
  \begin{aligned}
    \bigabs{J_4} &\le c\,\norm{a}_{C^{1,1}}\,\intO \abs {\trap
      {\pi}} \big (\abs{(\nabla \xi )\bpsi} + \abs{\xi \nabla \bpsi}\big
    )\, d\bx
    \\
    &\le  \epsilon \,2 (1+C_{P}) \norm{\xi}_{1,\infty}^2\intO
    \abs{\nabla \bpsi}^2 \, d\bx + c(\epsilon ^{-1})
    \,\norm{a}_{C^{1,1}}^2 \hspace*{-3mm} \int_{\Omega\cap \spt
      \xi}\hspace*{-3mm} \abs {\pi}^2 \, d\bx, 
  \end{aligned}
\end{equation}
where we used \Poincare's inequality and a translation
argument for $\pi$. Next, 
\begin{equation}
\label{eq:j5}
\begin{aligned}
  \bigabs{J_5} &\le \epsilon \, \intO \abs{\difn (\nabla \xi \cdot
    \bpsi)}^2 \, d\bx + c(\epsilon ^{-1}) \hspace*{-3mm}
  \int_{\Omega\cap \spt \xi}\hspace*{-3mm} \abs {\pi}^2 \, d\bx 
  \\
  &\le \epsilon \, \intO \abs{\nabla (\nabla \xi \otimess \bpsi)}^2 \,
  d\bx + c(\epsilon ^{-1}) \hspace*{-3mm} \int_{\Omega\cap \spt
    \xi}\hspace*{-3mm} \abs {\pi}^2\, d\bx 
  \\
  &\le \epsilon \,2( 1+C_{P})\norm{\xi}_{2,\infty} ^2
   \intO \abs{\nabla \bpsi}^2 \,
  d\bx+ c(\epsilon ^{-1}) \hspace*{-3mm} \int_{\Omega\cap \spt
    \xi}\hspace*{-3mm} \abs {\pi}^2 \, d\bx\,, 
\end{aligned}
\end{equation}
where we also used~\eqref{eq:diff} and \Poincare's
inequality. Finally,
\begin{equation}
  \label{eq:j6}
  \begin{aligned}
    \bigabs{J_6} &\le \epsilon \, \intO \abs{\difn (\xi \bpsi)}^2 \,
    d\bx + c(\epsilon ^{-1}) \hspace*{-3mm} \int_{\Omega\cap \spt
      \xi}\hspace*{-3mm} \abs {\ff}^2 \, d\bx 
    \\
    &\le \epsilon \, \intO \abs{\nabla (\xi \bpsi)}^2 \, d\bx +
    c(\epsilon ^{-1}) \hspace*{-3mm} \int_{\Omega\cap \spt
      \xi}\hspace*{-3mm} \abs {\ff}^2\, d\bx 
    \\
    &\le \epsilon \, 2(1+C_{P}) \norm{\xi}_{1,\infty}^2 \intO
    \abs{\nabla \bpsi}^2 \, d\bx+ c(\epsilon ^{-1}) \hspace*{-3mm}
    \int_{\Omega\cap \spt \xi}\hspace*{-3mm} \abs {\ff}^2 \, d\bx\,,
  \end{aligned}
\end{equation}
where we  used~\eqref{eq:diff} and \Poincare's inequality.
In the estimates~\eqref{eq:j1}--\eqref{eq:j6} we choose the following
constant $\epsilon>0$
\begin{equation*}
  \epsilon=\frac{1}{24C_{0}(1+C_{P})\norm{\xi}_{2,\infty}^2},
\end{equation*}
and we absorb all terms with $\epsilon$ in the term with $C_0^{-1}$
in~\eqref{eq:tangpres3}. We thus we obtain
from~\eqref{eq:tangpres1},~\eqref{eq:tangpres2},~\eqref{eq:tangpres3},~\eqref{eq:j6},
and~\eqref{eq:tang1} that
\begin{equation*}
    \hspace*{-2mm}\intO \xi^2 \bigabs{ \difp \pi}^2 \, d\bx
\le c( \delta^{p-2},\norm{\xi}_{2,\infty},\norm{a}_{C^{1,1}})
    \hspace*{-3mm} \int_{\Omega\cap \spt \xi}\hspace*{-3mm}
    \abs{\ff}^2 \! + \! \abs{\pi}^2 \! + \!  \abs{\nabla \bu}^2 \! + \!  \phi(\abs{\nabla
      \bu}) \, d\bx\,,     \hspace*{-2mm}
\end{equation*} 
where the right-hand side is finite, since~\eqref{eq:7}
and~\eqref{eq:est-tlak} imply 
\begin{equation}
  \intO \xi^2 \bigabs{ \difp \pi}^2 \, d\bx
  \le M_1( \delta^{p-2},\norm{\xi}_{2,\infty},\norm{a}_{C^{1,1}})
  (1+    \norm{\ff}^2)\,.   \label{eq:tang2}
\end{equation} 
From this
and~\eqref{eq:2a} we thus obtain the
estimate~\eqref{eq:tang-final-pres}.
\end{proof} 
%
\begin{remark}
  \label{rem:interior_reg1}
  The same procedure as in the proof of
  Proposition~\ref{prop:tangential1}, with many simplifications, can
  be done in the interior of $\O$ for divided differences in all
  directions $\be^i$, $i=1,2,3$. By choosing $h\in\Big(0,\frac{1}{2}
  \dist(\spt\xi_{00},\partial\Omega)\Big)$ this leads to
  \begin{equation}
    \label{inreg}
    \intO \xi_{00}^2 \bigabs{ \nabla^2 \bu }^2 +\xi_{00}^2 \bigabs{
      \nabla \F (\bD\bu) }^2 + \phi (\xi_{00} \abs{ \nabla^2
      \bu})\,d\bx \le c( \norm{\xi_{00}}_{2,\infty}) \big (
    1+\norm{\bff}_2^2
    \big )\,,
  \end{equation}
  \begin{align}
    \label{inreg-pres}
  \intO \xi^2_{00} \bigabs{ \nabla \pi }^2 \, d\bx \le c(\delta^{p-2},
  \norm{\xi_{00}}_{2,\infty}) \big ( 1+ \norm{\bff}_2^2
  \big )\,,
  \end{align}
  where $\xi_{00}$ is any cut-off function with compact support
  contained in $\Omega$.
\end{remark}

The following fundamental result derives immediately from the interior
regularity.
  \begin{corollary}\label{cor:a.e.}
    Under the same hypotheses of Theorem~\ref{thm:MT1} we obtain that
    $\bF(\bD)\in W^{1,2}_{\loc}(\Omega)$ and we know from~\eqref{inreg}
    and~\eqref{inreg-pres} that $\bu\in W^{2,2}_\loc(\O)$, \linebreak
    $\pi \in W^{1,2}_\loc (\Omega)$. This implies --in particular--
    that the equations of system~\eqref{eq} hold almost everywhere in
    $\Omega$
  \end{corollary}
  The above corollary implies that all pointwise calculations we are
  going to perform in the next section to estimate the remaining
  derivatives are justified.
\subsection{Regularity in the normal direction}
\label{sec:nor}
In this section we obtain global information about the regularity of
$\nabla^2\bu$ and $\nabla\pi$, by combining the information about
tangential regularity~\eqref{eq:tang-final} and the fact that the
couple $(\bu,\pi)$ satisfies almost everywhere the system~\eqref{eq}.
We use methods applied also in~\cite{hugo-nonflat,hugo-petr-rose,mnr}
for $p>2$ and and which have been previously used in the case $p<2$
also in~\cite{Ber-Gri2-2016} (for the problem without pressure) and
in~\cite{BdV2009}.
The main result of this section is the following proposition.
\begin{Proposition}
  \label{prop:normal1}
  Let the assumptions of Theorem~\ref{thm:MT1} be satisfied and let
  the local description $a_P$ of the boundary and the localization
  function $\xi_P$ satisfy $(b1)$--\,$(b3)$ and $(\ell 1)$
  (cf.~Section~\ref{sec:bdr}). Then, there exist a constant $C_2>0$
  and functions $M_2, M_3 $ such that for every $P\in \partial \Omega$
  \begin{align}
    \intO \xi^2_P \bigabs{ \nabla^2 \bu }^2 \,d\bx 
    &\le M_2 \,\big ( 1+\norm{\bff}_2^2
      \big )\,,\label{eq:norm-final-1}
    \\
    \intO \xi^2_P \bigabs{ \nabla\F(\bD\bu) }^2\,d\bx 
    &\le   M_{2}\,\delta^{p-2} \, \big (
      1+\norm{\bff}_2^2
      \big )\,,\label{eq:norm-final-2}
  \end{align}
  where
  $M_2:=1 +M_1(\delta^{p-2},
  \norm{\xi_P}_{2,\infty}, \norm{a_P}_{C^{1,1}}) + (1+\delta^{2(p-2)})M_0(
  \norm{\xi_P}_{2,\infty}, \norm{a_P}_{C^{2,1}})$, and
  \begin{equation}
     \label{eq:norm-final-pres}
    \intO \xi^2_P \bigabs{ \nabla\pi }^2 \, d\bx \le
    M_3\, \big ( 1+\norm{\bff}_2^2
    \big  )\,,
  \end{equation}
  where $M_3:=1 + 2(1+\kappa_1^2\,\delta^{2(p-2)})
  M_2$, provided that $r_P <C_2$.
\end{Proposition}
\begin{proof}
  Let us consider the first two equations in~\eqref{eq}. In order to
  better separate variables concerning tangential and normal
  directions we use the following convention: Greek low-case letters
  concern only tangential derivatives while Latin low-case concern all
  of them, that is     $\alpha,\beta,\gamma,\sigma=1,2$, and 
    $i,j,k,l=1,2,3$, %
    with the usual convention of summation over repeated
    indices. Moreover, we denote a two-dimensional vector with
    components $(b_1,b_2)^\top$ by $\boldsymbol { \mathfrak b}$. Let
    us put in evidence the equations for the variables
    $\bS_{\alpha 3}$, by re-writing the $\alpha$-th equation
    from~\eqref{eq} as follows:
  \begin{equation*}
    -\partial_3 \bS_{\alpha 3}=f_\alpha+\partial_\alpha
    \pi+\partial_\beta \bS_{\alpha \beta}\qquad \textrm {a.e.\ in }\Omega.
  \end{equation*}
  We previously estimated $\partial_\tau \bD\bu$ and we need now to
  estimate $\partial_3\bD_{\gamma3}$, to recover all second order
  derivatives of $\bu$.  We re-write explicitly
  (component-wise)~$\eqref{eq}_{\alpha}$ as follows:
  \begin{equation}
    \label{eq:separation_variables}
    \begin{aligned}
      -&\partial_{\gamma 3} \bS_{\alpha 3}\partial_3 \bD_{\gamma 3}
      -\partial_{3 \gamma } \bS_{\alpha 3}\partial_3 \bD_{ 3 \gamma}
      \\
      & =f_\alpha+ \partial_\alpha \pi+\partial_{3 3}\bS_{\alpha
        3}\partial_3 \bD_{3 3}+\partial_{\gamma \sigma}\bS_{\alpha
        3}\partial_3 \bD_{\gamma \sigma}+ \partial_{k l}\bS_{\alpha
        \beta}\partial_\beta\bD_{k l}=:\mathfrak {f}_\alpha\,,
    \end{aligned}
  \end{equation}
  by keeping only the above two terms in the left-hand side.  The
  system~\eqref{eq:separation_variables} can be written as
  the algebraic linear system
  \begin{equation}
    \label{eq:linear_system}
    -2  A_{\alpha\gamma}b_\gamma=\mathfrak{f}_\alpha\qquad\textrm
    {a.e.\ in }\Omega\,, 
  \end{equation}
  where $ A_{\alpha \gamma}:=\partial_{\gamma 3}\bS_{\alpha 3}$, and
  $\mathfrak {b}_\gamma:=\partial_3 \bD_{\gamma 3}$. This is obtained by recalling
  that $\bD$ is symmetric, hence $\bD_{\gamma 3}=\bD_{3 \gamma}$, and
  by observing that
  $\partial_{\gamma 3}\bS(\bD)=\partial_{3 \gamma}\bS(\bD)$, since
  $\bS$ depends only on the symmetric tensor $\bD$.  

  We will estimate the quantity $\mathfrak b_\gamma =\partial_{3}\bD_{\gamma 3}$
  in terms of tangential derivatives, estimated in the previous
  section, and then extract from it information on the normal
  derivatives. A similar technique was employed in the above cited
  references, but there immediately a system for the normal
  derivatives $\partial ^2_{33}\bu$ has been derived, which led to the
  restriction $p>\frac32$, even in the flat case.  

  Multiplying~\eqref{eq:linear_system} pointwise a.e. by $-\mathfrak
  b_\alpha$ and summing over $\alpha=1,2$ we get, also using the
  structure of $\bS$,
\begin{equation*}
  2\big( \kappa_0(2)+\kappa_0(p)\,\phi''(|\bD\bu|)\big)|\boldsymbol { \mathfrak b}|^2
  \leq 2A_{\alpha \gamma}\mathfrak b_\gamma \mathfrak 
  b_\alpha\leq|\boldsymbol { \mathfrak f}|\,|\boldsymbol { \mathfrak
    b}|\quad\textrm {a.e.\ in }\Omega.   
\end{equation*}
To handle $\boldsymbol { \mathfrak f}$ we prove identities and estimates valid almost
everywhere in $\Omega_P$.  We start with the following identity for
the pressure:
\begin{equation*}
  \partial_\alpha \pi=\partial_{\tau_\alpha}\pi-\partial_\alpha
  a\,\partial_3 \pi\qquad\qquad \textrm {a.e.\ in }  \Omega_P\,.
\end{equation*}
Concerning the third term from the right-hand side
of~\eqref{eq:separation_variables} we get a.e.~in $\Omega_P$, by using
the solenoidal constraint and the formula for tangential derivatives, 
  \begin{equation*}
    \begin{aligned}
      \partial_3\bD_{3 3}&=-\partial_3 \bD_{\alpha\alpha}=-\partial_3
      (\nabla \bu)_{\alpha\alpha}=-\partial_\alpha\partial_3 u^\alpha
      =-\partial_{\tau_\alpha}\partial_3 u^\alpha+\partial_\alpha
      a\,\partial_{33}^2 u^\alpha.
    \end{aligned}
  \end{equation*}
  Concerning the fourth term from the right-hand side
  of~\eqref{eq:separation_variables} we observe that a.e.~in $\Omega_P$
  \begin{equation*}
    \partial_3\bD_{\gamma\sigma}=\frac{1}{2}\left(\partial^2_{\gamma3}u^\sigma
      +\partial_{\sigma3}^2u^\gamma\right)=
    \frac{1}{2}\left(\partial_{\tau_\gamma}\partial_3
      u^\sigma-\partial_\gamma a\, \partial^2_{33}u^\sigma+\partial_{\tau_\sigma}\partial_3
      u^\gamma-\partial_\sigma a\, \partial^2_{33}u^\gamma\right).
  \end{equation*}
  The fifth term from the right-hand side
  of~\eqref{eq:separation_variables} is handled recalling that
  \begin{equation*}
    \partial_\beta\bD_{k l}=\partial_{\tau_\beta}\bD_{k
      l}-\partial_\beta a\, \partial_3\bD_{kl}\qquad \textrm{a.e.\
    in } \Omega_P.
  \end{equation*}
  Collecting all these identities we obtain that a.e.~in $\Omega_P$
  the right-hand side $\boldsymbol { \mathfrak f}$
  of~\eqref{eq:separation_variables} can be bounded as follows :
  \begin{equation*}
      |\boldsymbol { \mathfrak f}| 
      \leq c
      \left(|\bff|+|\partial_\tau \pi|+\|\nabla
        a\|_{\infty}|\partial_3
        \pi|+(1+\phi''(|\bD\bu|))\left(|\partial_\tau\nabla \bu|+\|\nabla
          a\|_{\infty}|\nabla^2 \bu|\right)\right),
  \end{equation*}
  where 
  the constant $c$ depends only on 
  the characteristics of $\bS$.
  To estimate the partial derivative $\partial_3\pi$ we use again the
  equations to write pointwise in $\Omega$ that
  $\partial_{3}\pi=-f_{3}-\partial_{j}\bS_{3 j}$ and hence to obtain
  \begin{equation}
    |\partial_{3}\pi|\leq
    |\bff|+c\big(1+\phi''(|\bD\bu|)\big)|\nabla ^{2}\bu|\qquad \textrm {a.e.\
    in }\Omega_P\,.\label{eq:pi3}
  \end{equation}
  Collecting the estimates and dividing both sides by
  $|\boldsymbol{\mathfrak b}|\not=0$ (when it is zero there is nothing
  to prove) we obtain
  \begin{equation*}
    \begin{aligned}
      &\left( 1+\phi''(|\bD\bu|)\right)|\boldsymbol { \mathfrak b}|
      \\
      &\leq c \left(|\bff|+|\bff|\|\nabla a\|_{\infty}+|\partial_\tau
        \pi|+(1+\phi''(|\bD\bu|))\left(|\partial_\tau\nabla
          \bu|+\|\nabla a\|_{\infty}|\nabla^2 \bu|\right)\right),
    \end{aligned}
  \end{equation*}
  We now identify the ``normal''\footnote{It is not the normal
    derivative at all points on $\partial\Omega\cap \partial\Omega_P$,
    but only exactly at $\bx=P$} derivative $\partial_{33}u^{\alpha}$
  from the left-hand side by observing that
  \begin{equation*}
    \begin{aligned}
      \mathfrak b_{\alpha}&=\frac{1}{2}\left(\partial_{3 \alpha}^2
        u^{3}+\partial_{33}^2u^{\alpha}\right)=
      \frac{1}{2}\left(-\partial_{\alpha} \bD_{\beta
          \beta}+\partial_{33}^2u^{\alpha}\right)=
      \frac{1}{2}\left(-\partial^2_{\alpha \beta}u^\beta+\partial^2_{33}u^{\alpha}\right)
      \\
      &=\frac{1}{2}\left(-\partial_{\tau_{\alpha}}\partial_{\beta}
        u^{\beta}+\partial_{\alpha}
        a\,\partial_{3\beta}^2u^{\beta}+\partial^2_{33}u^{\alpha}\right).
    \end{aligned}
  \end{equation*}
  Consequently, we can write (in terms of
  $\tilde{\mathfrak b}_{\alpha}:=\partial^2_{33}u^{\alpha}$)
  \begin{equation*}
    |\boldsymbol { \mathfrak b}|\geq2|\tilde{\boldsymbol{\mathfrak
        b}}|-|\partial_{\tau}\nabla\bu|-\|\nabla 
    a\|_{\infty}|\nabla ^{2}\bu|\quad\textrm{a.e. in }\Omega_P, 
  \end{equation*}
  from which we can show the following fundamental estimate
  \begin{equation*}
    \begin{aligned}
      &\left( 1+\phi''(|\bD\bu|)\right)|\tilde {\boldsymbol { \mathfrak b}}|
      \\
      &\leq c \left(|\bff|+|\bff|\|\nabla a\|_{\infty}+|\partial_\tau
        \pi|+(1+\phi''(|\bD\bu|))\left(|\partial_\tau\nabla
          \bu|+\|\nabla a\|_{\infty}|\nabla^2 \bu|\right)\right),
    \end{aligned}
  \end{equation*}
  which is valid almost everywhere in $\Omega_P$. 
  %
  Next, we observe that adding to both sides, for $\alpha=1,2$ and
  $i,k=1,2,3$ the term
  \begin{equation*}
    \left(1+\phi''(|\bD\bu|)\right)\big(|\partial_\alpha\partial_i
    u^k|+|\partial_{33}^2 u^3|\big)
  \end{equation*}
  the left-hand side equals $\left(1+\phi''(|\bD\bu|)\right)|\nabla ^{2}\bu|$.
  To handle the terms we added in the right-hand side, we observe,
  that $\partial_{33} u^{3}=-\partial_{3\alpha}u^{\alpha}$ implies (by the
  solenoidal constraint)
  \begin{equation*}
    \begin{aligned}
      \partial_\alpha\partial_i
      u^k&=\partial_{\tau_\alpha}\partial_i u^k-\partial_\alpha
      a\,\partial^2_{3i} u^k,
      \\
      \partial_{33}^2u^{3}&=-\partial_{\tau_\alpha}\partial_3
      u^\alpha+\partial_\alpha a\,\partial^2_{33} u^\alpha\,.
    \end{aligned}
  \end{equation*}
  Hence, there exists a constant $C_1$, depending only on the
  characteristics of $\bS$, such that a.e.~in $\Omega_P$ it holds
  \begin{align}
    &\left( 1+\phi''(|\bD\bu|)\right)|\nabla ^2\bu|     \label{eq:estimate_normal_1}
    \\
    &\leq C_1 \left(|\bff|+|\bff|\|\nabla a\|_{\infty}+|\partial_\tau
      \pi|+(1+\phi''(|\bD\bu|))\left(|\partial_\tau\nabla
      \bu|+\|\nabla a\|_{\infty}|\nabla^2
      \bu|\right)\right)\,. \notag 
  \end{align}
  Next, we choose the open sets $\Omega_{P}$ small enough (that is we
  choose the radii $R_{P}$ small enough) in such a way that
  \begin{equation}
    \label{eq:smallness-a}
    \|\nabla a_{P}(x)\|_{L^{\infty}(\Omega_{P})}\le    r_{P}\leq \frac{1}{2C_1}=:C_2\,.
  \end{equation}
  Thus, we can absorb the last term from~\eqref{eq:estimate_normal_1} in
  the left-hand side, which yields 
  \begin{equation}
    \label{eq:estimate_normal_final}
    \begin{aligned}
      & \left( 1+\phi''(|\bD\bu|)\right)|\nabla ^{2}\bu|
      \\
       &
       \qquad 
       \leq c\left(|\bff|+|\partial_{\tau}\pi|+\big(1+
         \phi''(|\bD\bu|)\big)|\partial_{\tau}\nabla\bu|\right)
       \quad \textrm{a.e.\ in }\Omega_P.
    \end{aligned}
  \end{equation}
  It is at this point that we use the special features of the stress
  tensor which is the sum of the quadratic part with one with
  $(p,\delta)$-structure for $p<2$.  Hence, neglecting the second term
  on the left-hand side, which is non-negative,  raising the
  remaining inequality to the power $2$, and multiplying
  by the cut-off function $\xi_P^2$, we obtain
  \begin{equation}
    \label{eq:6}
    |\nabla^2\bu|^2\,\xi_P^2\leq
    c\left(|\bff|^2+|\partial_{\tau}\pi|^2+(1+\delta^{2(p-2)})
      |\partial_\tau\nabla\bu|^2
    \right)\,\xi_P^2    \quad  \textrm{a.e.\ in }\Omega_P.
  \end{equation}
  Corollary~\ref{cor:a.e.} implies that the left-hand side is
  measurable, while Proposition~\ref{prop:tangential1} implies that
  the right-hand side belongs to $L^1(\Omega_P)$. Thus,
  Proposition~\ref{prop:tangential1} yields
  estimate~\eqref{eq:norm-final-1}.
  Next, we observe that (see~\cite[Lemma~3.8]{bdr-7-5})
  \begin{equation*}
    |\nabla\bF(\bD)|^2\leq (\delta+|\bD|)^{p-2}|\nabla \bD \bu|^2
    \leq {\delta^{p-2}}|\nabla^2\bu|^2\,, 
  \end{equation*}
  which together with~\eqref{eq:norm-final-1} yields~\eqref{eq:norm-final-2}. 
  The estimate on $\nabla\pi$ is obtained from equation~\eqref{eq},
  since 
  \begin{equation*}
    |\nabla\pi|\leq
    |\bff|+\big(1+\kappa_1\phi''(|\bD\bu|)\big)|\nabla^{2}\bu| \,,
  \end{equation*}
  where we also used~\eqref{eq:ass_S}. This
  and~\eqref{eq:norm-final-1} yields~\eqref{eq:norm-final-pres}.
\end{proof}

\begin{proof}[Proof of Theorem~\ref{thm:MT1}]
  For every $P\in\bou$ we choose a local description $a_P$ of the
  boundary satisfying $(b1)$--\,$(b3)$ (cf.~Section~\ref{sec:bdr})
  with $r_P<C_2$. Note that
  \begin{equation}
    \label{eq:covering}
    \bou \subset \bigcup_{P\in\bou}\widetilde{\O}_P\,,
  \end{equation}
  where $\widetilde{\Omega}_P=\frac{1}{2}\Omega_P$,
  recall~\eqref{eq:scaled_omega_P}.  Since $\bou $ is compact there
  exists a finite sub-covering $\{\widetilde \Omega_P,\ P\in
  \II\}$. Next, we choose a set $\Omega_0\subset \subset \Omega$ such
  that $\dist(\Omega_0,\partial\Omega)<\frac{1}{16}
  \min\{R_P,{P\in\II}\}$.  Associated to the covering of $\Omega$ made
  by $\{\Omega_0,\,\O_P\}_{P\in\II}$ we consider a set of smooth
  non-negative functions $\{\xi_0,\xi_P\}_{P\in \II}$, where $\xi_P$
  satisfy $(\ell1)$ and $\xi_0 $ satisfies $\spt
  \xi_0\subset\Omega_0$,
  $\dist(\spt\xi_0,\partial\Omega)<\frac{1}{8}\min\{R_P,{P\in\II}\}$,
  and $\xi_0(\bx)=1$ for all $\bx$ with $\dist (\bx,
  \bou)>\frac{1}{4}\min\{R_P,{P\in\II}\}$. Observe that, by
  construction the set $\{\bx:\ \xi_0=1\}\cup_{P\in \II} \{\bx:\
  \xi_P=1\}$ covers all ${\Omega}$.  Since the covering
  $\{\Omega_0,\,\O_P\}_{P\in\II}$ is finite, the evaluation of the
  functions $M_i$, $i=0,1,2,3,$ yields a finite fixed constant. Thus,
  Proposition~\ref{prop:tangential1} and
  Proposition~\ref{prop:normal1} applied to the finite covering
  $\{\O_0,{\O}_P\}_{P\in\II}$ of $\Omega$ yields immediately the
  assertions of Theorem~\ref{thm:MT1}.
\end{proof}

\section{Proof of Theorem~\ref{thm:MT2}}
\label{sec:p-stokes}
In this section we treat the problem with the principal part having
$(p,\delta)$-struct\-ure. Many calculations are similar to those of the previous
section, hence we recall them and mainly explain the differences
arising in the treatment of tangential and normal derivatives. We
assume that $\T$ satisfies the assumption of Theorem~\ref{thm:MT2},
i.e.,~$\bS$ has $(p,\delta)$-structure. Moreover, $\ff$ belongs to
$L^{p'}(\Omega)$. Due to Remark~\ref{rem:equi-norm} this is equivalent
to $\ff \in L^{\phi^*}(\Omega)$.

As in Section~\ref{sec:Stokes} we easily obtain the existence of a
unique  $\bu \in W^{1,p}_{0,\divo}(\O)$ satisfying the weak
formulation, i.e.~for all $\bv\in W^{1,p}_{0,\divo}(\O)$
  \begin{equation*}
    \intO\T(\bD\bu)\cdot\bD\bv\, d\bx=\intO \ff \cdot \bv\,
    d\bx\,, 
  \end{equation*}
and the a-priori estimate 
\begin{align}
  \label{eq:7p}
  \intO \varphi(\abs{\nabla \bu }) + \abs{\bF(\bD\bu)}^2 \, d\bx \le c\, \intO 
  \phi^*(\abs{\ff}) \, d\bx\,.
\end{align}
In this section all constants can depend on the characteristics of
$\bS$, on $\diameter (\Omega)$, $\abs{\Omega}$, on the space
dimension, on the John constants of $\Omega$, and on $\delta_0$
(cf.~Remark~\ref{rem:delta0}). All these dependencies will not
mentioned explicitly, while the dependence on other quantities is made
explicit.

To the weak solution $\bu$ there exists a unique associated pressure
$\pi\in L^{p'}_0(\O)$ satisfying for all $\bv\in W^{1,p}_0(\O)$
\begin{equation}
  \label{eq:eq-tlakp}
  \intO\T(\bD\bu)\cdot\bD\bv-\pi\diver\bv \, d\bx =\intO \ff \cdot
  \bv\, d\bx\,. 
\end{equation}
Using the properties of $\bS$, Lemma~\ref{lem:ism} and Young's
inequality we thus obtain
\begin{align}\label{eq:est-tlakp}
  \intO\phi^*(\abs{\pi })\, d\bx \le c\,\intO \phi^*(  \abs{\ff}) \, d\bx \,.
\end{align}
\subsection{Regularity in tangential  directions and in the interior}
\label{sec:tanp}
We start again with the regularity in tangential directions.
In the same way as in Section~\ref{sec:Stokes} we derive the following result.
\begin{Proposition}
\label{prop:tangential2}
Let the assumptions of Theorem~\ref{thm:MT2} be satisfied and let the
local description $a_P$ of the boundary and the localization function
$\xi_P$ satisfy $(b1)$--\,$(b3)$ and $(\ell 1)$
(cf.~Section~\ref{sec:bdr}). Then, there exists functions
$M_4 \in C(\setR^{\ge 0}\times \setR^{\ge 0})$,
$M_5 \in C(\setR^{>0}\times\setR^{\ge 0}\times \setR^{\ge 0})$ such that for
every $P\in \partial \Omega$
\begin{equation}
  \label{eq:tang-finalp}
  \begin{aligned}
    \intO \xi_P^2 \bigabs{ \partial _\tau \F (\bD\bu) }^2 + &\,\phi
    (\xi_P \abs{\partial _\tau \nabla \bu})+\phi (\xi_P
    \abs{\nabla\partial _\tau \bu})\, d\bx
    \\
    &\le M_4( \norm{\xi_P}_{2,\infty},\norm{a_P}_{C^{2,1}})
    \Big ( 1+ \intO  \phi^*(\abs{\ff}) \, d\bx\Big )\,, 
  \end{aligned}
\end{equation}
and 
\begin{align}
  \label{eq:tang-final-presp}
  \intO \xi_P^2 \bigabs{ \partial_\tau \pi }^2 \, d\bx \le
  M_5(\delta^{p-2}, \norm{\xi_P}_{2,\infty},\norm{a_P}_{C^{2,1}}) 
  \Big ( 1+ \intO  \phi^*(\abs{\ff}) \, d\bx\Big )\,,
\end{align}
where $\norm{a_P}_{C^{k,1}}$ means
$\norm{a_P}_{C^{k,1}(\overline{B^2_{3/4R_P}(0)})}$, for $k=1,2$.
\end{Proposition}
%
%
\begin{proof}
  As in the previous section we fix $P\in \partial \Omega$ and in
  $\O_P$ use $\xi:=\xi_P$, $a:=a_P$, 
  while $h\in(0,\frac{R_{P}}{16})$, as in Section~\ref{sec:bdr}. In
  the same way as in Section~\ref{sec:tan}, replacing
  $W^{1,2}_{0,\divo}(\O)$ by $W^{1,p}_{0,\divo}(\O)$ and using
  Lemma~\ref{lem:p-est} for $p\in (1,2)$ only, we take as test
  function $\bv=\difn{(\xi^2\difp\bu_{|\widetilde{\Omega}_P})}$, to
  obtain~\eqref{eq:dtTx2} and
  \begin{equation}
    \label{eq:tangp}
    \begin{aligned}
      \intO \xi^2 \bigabs{ \difp \F (\bD\bu) }^2& + \phi (\xi \,
      \abs{\difp\nabla \bu}) + \phi (\xi \,\abs{\nabla \difp \bu})\,
      d\bx
      \\
      &\le c \, \sum_{j=1}^{15} I_j+ c(
      \norm{\xi}_{1,\infty},\norm{a}_{C^{1,1}}) 
      \hspace*{-3mm}
      \int_{\Omega\cap \spt \xi}\hspace*{-3mm} \phi \big (\abs{
        \nabla\bu }\big ) \, d\bx\,.
    \end{aligned}
  \end{equation}
  Note, that in Section~\ref{sec:tan} all terms in $I_1,\ldots,I_7$
  resulting from $\bS^0$ have been absorbed in the corresponding terms
  in~\eqref{eq:tang} coming from $\bS^0$. The same for the terms in
  $I_1,\ldots,I_7$ resulting from $\bS^1$ which have been absorbed in
  the corresponding terms in~\eqref{eq:tang} coming from $\bS^1$,
  only.

  This implies that we can now treat $\bS$ in $I_1,\ldots,I_7$ here as
  $\bS^1$ in the corresponding terms $I_1,\ldots,I_7$ in
  Section~\ref{sec:Stokes} and obtain then
  \begin{align}
    &\bigabs{I_1} \le c\, (\varepsilon^{-1},\norm{a}_{C^{1,1}})
    \hspace*{-3mm} \int_{\Omega\cap \spt \xi}\hspace*{-3mm}
    \phi(\abs{\bD\bu}) \, d\bx + \varepsilon \intO \phi
    (\xi\abs{\difp\nabla \bu}) \, d\bx \,,\label{eq:i1p}
    \\
    &\bigabs{I_2 + I_3} \le c\,
    (\norm{\xi}_{1,\infty},\norm{a}_{C^{1,1}}) \hspace*{-3mm}
    \int_{\Omega\cap \spt\xi}\hspace*{-3mm} \phi(\abs{\nabla\bu}) \,
    d\bx\,,\label{eq:i23p}
    \\
    &\bigabs{I_4}\le c\, (\norm{a}_{C^{2,1}}) \hspace*{-3mm}
    \int_{\Omega\cap \spt \xi}\hspace*{-3mm} \phi(\abs{\nabla\bu}) \,
    d\bx\,, \label{eq:i4p}
    \\
    &\bigabs{I_5}\le c(\epsilon ^{-1},\norm{\xi}_{2,\infty})
    \hspace*{-3mm} \int_{\Omega\cap \spt \xi}\hspace*{-3mm}
    \phi(\abs{\nabla\bu}) \, d\bx + \varepsilon
    4\norm{\xi}_{1,\infty}^2 \hspace*{-3mm} \int_{\Omega\cap \spt
      \xi}\hspace*{-3mm} \phi(\abs{\xi\nabla\difp\bu}) \, d\bx
    \,,\label{eq:i5p}
    \\
    &\bigabs{I_6} \le c\, (\norm{\xi}_{1,\infty},\norm{a}_{C^{1,1}})
    \hspace*{-3mm} \int_{\Omega\cap \spt \xi}\hspace*{-3mm}
    \phi(\abs{\nabla\bu}) \, d\bx\,, \label{eq:i6p}
    \\
    &\bigabs{I_7} \le c\, (\varepsilon^{-1},\norm{a}_{C^{1,1}})
    \hspace*{-3mm} \int_{\Omega\cap \spt
      \xi}\hspace*{-3mm}\phi(\abs{\bD\bu}) \, d\bx + \varepsilon \intO
    \phi (\xi\abs{\difp\nabla \bu}) \, d\bx \,.\label{eq:i7p}
  \end{align}
  Also the terms with the pressure and the external force are
  estimated similarly to the corresponding terms in
  Section~\ref{sec:Stokes}. Since we know that $\pi \in L_0^{p'}(\O)$
  instead of $\pi \in L_0^2(\O)$ we use Young's inequality with $\phi$
  instead of $2$. In that way we obtain %
%
  \begin{align}
    \bigabs{I_8} & \le c\, (\norm{a}_{C^{2,1}}) \hspace*{-3mm}
    \int_{\Omega\cap \spt \xi}\hspace*{-3mm} \phi^*(\abs{\pi}) +
    \phi(\abs{\nabla \bu})\, d\bx\,, \label{eq:i8p}
    \\
    \bigabs{I_9 +I_{10}}&\le c\,
    (\norm{a}_{C^{1,1}},\norm{\xi}_{1,\infty}) \hspace*{-3mm}
    \int_{\Omega\cap \spt \xi}\hspace*{-3mm} \phi^*(\abs{\pi}) +
    \phi(\abs{\nabla \bu}) \, d\bx\,, \label{eq:i910p}
    \\
    \bigabs{I_{11}}&\le c(\epsilon ^{-1},\norm{\xi}_{2,\infty})
    \hspace*{-3mm} \int_{\Omega\cap \spt \xi}\hspace*{-3mm}
    \phi^*(\abs{\pi})+ \phi(\abs{\nabla \bu})\, d\bx \notag
    \\
    &\hspace*{30mm} + \varepsilon \max \big (1,4\norm{\xi}_{1,\infty}^2\big ) \hspace*{-3mm}
    \int_{\Omega\cap \spt \xi}\hspace*{-3mm}\phi(\xi\abs{\nabla
      \difp\bu})\, d\bx\,,\label{eq:i11p}
    \\
    \bigabs{I_{12}} &\le c\, (\epsilon^{-1},\norm{a}_{C^{1,1}})
    \hspace*{-3mm} \int_{\Omega\cap \spt \xi}\hspace*{-3mm}
    \phi^*(\abs{\pi})\, d\bx + \epsilon \int_{\Omega}
    \phi(\xi\abs{\difp\nabla \bu}) \, d\bx\,, \label{eq:i12p}
    \\
    \bigabs{I_{13}} &\le c\, (\epsilon^{-1},\norm{a}_{C^{1,1}})
    \hspace*{-3mm} \int_{\Omega\cap \spt \xi}\hspace*{-3mm}
    \phi^*(\abs{\pi}) + \phi(\abs{\nabla \bu}) \,
    d\bx\,, \label{eq:i13p}
    \\
    \bigabs{I_{14}} &\le c\, (\epsilon^{-1},\norm{a}_{C^{1,1}})
    \hspace*{-3mm} \int_{\Omega\cap \spt \xi}\hspace*{-3mm}
    \phi^*(\abs{\pi})\, d\bx + \epsilon \int_{\Omega}
    \phi(\xi\abs{\difp\nabla \bu}) \, d\bx\,, \label{eq:i14p}
    \\
    \bigabs{I_{15}} &\le c(\epsilon ^{-1},\norm{\xi}_{1,\infty})
    \hspace*{-3mm} \int_{\Omega\cap \spt
      \xi}\hspace*{-3mm}\phi^*(\abs{\ff})+ \phi(\abs{\nabla \bu})\,
    d\bx + \varepsilon \hspace*{-3mm} \int_{\Omega\cap \spt
      \xi}\hspace*{-3mm}\phi(\xi \abs{\nabla \difp\bu})\,
    d\bx\,, \label{eq:i15p}
  \end{align}
  where we also used~\eqref{eq:est-l} for the treatment of $I_{11}$
  (cf.~estimate~\eqref{eq:i5}). 
  Choosing in the estimates~\eqref{eq:i1p}--\eqref{eq:i15p} $\epsilon$
  the constant $\epsilon>0$
  \begin{equation*}
    \epsilon=\frac{1}{56\norm{\xi}_{1,\infty}^2},
  \end{equation*}
  we can absorb all terms involving tangential increments of
  $\nabla\bu$ in the left-hand side of~\eqref{eq:tang}, obtaining the
  following fundamental estimate
  \begin{equation}
    \label{eq:tang1p}
    \begin{aligned}
      \intO & \xi^2 \bigabs{ \difp \F (\bD\bu) }^2 + \phi (\xi
      \abs{\difp\nabla \bu}) + \phi (\xi \abs{\nabla \difp \bu})\,
      d\bx
      \\
      &\le c( \norm{\xi}_{2,\infty},\norm{a}_{C^{2,1}})
      \hspace*{-4mm} \int_{\Omega\cap \spt \xi}\hspace*{-4mm}
      \phi^*(\abs{\ff}) +\phi^*(\abs{\pi}) + \phi(\abs{\nabla \bu}) \,
      d\bx\,,
    \end{aligned}
  \end{equation}
  where the  right-hand side is finite (and independent of $h$)
  since~\eqref{eq:7p},~\eqref{eq:est-tlakp} imply 
  \begin{equation*}
    \begin{aligned}
      \intO & \xi^2 \bigabs{ \difp \F (\bD\bu) }^2 + \phi (\xi
      \abs{\difp\nabla \bu}) + \phi (\xi \abs{\nabla \difp \bu})\,
      d\bx
      \\
      & \le M_4( \norm{\xi}_{2,\infty},\norm{a}_{C^{2,1}}) \Big ( 1+
      \intO \phi^*(\abs{\ff}) \, d\bx\Big )\,. 
    \end{aligned}
  \end{equation*}
hence, by recalling ~\eqref{eq:2a} we  proved~\eqref{eq:tang-finalp}. 

\smallskip

To estimate the tangential derivatives of the pressure we proceed as
in Section~\ref{sec:Stokes}. The only difference is that the stress
tensor has no part with a $2$-structure. Even in this setting with pure
$(p,\delta)$-structure we are again able to show that $\partial_\tau
\pi\in L^2(\Omega\cap\spt\xi_P)$.  Starting from~\eqref{eq:tangpres1}
and neglecting from~\eqref{eq:tang2} the terms which resulted from
$\bS^0$, we arrive at the following inequality
\begin{equation}
  \label{eq:tang2p}
  \begin{aligned}
    \intO \xi^2 \bigabs{ \difp \pi}^2 \, d\bx &\le c(
    \delta^{p-2},\norm{\xi}_{2,\infty},\norm{a}_{C^{1,1}})
    \hspace*{-4mm} \int_{\Omega\cap \spt \xi}\hspace*{-4mm}
    \abs{\ff}^2 +\abs{\pi}^2 + \phi(\abs{\nabla \bu}) \, d\bx\,.
    \hspace*{-1mm}
  \end{aligned}
\end{equation} 
Young's inequality, Lemma~\ref{lem:change2} and $p<2$ yield
\begin{align*}
  \int_{\Omega\cap \spt \xi}\hspace*{-3mm} \abs{\ff}^2 +\abs{\pi}^2 \,
  d\bx \le c\, \Big ( (1+\delta^p)\abs{\Omega } + \hspace*{-3mm}
  \int_{\Omega\cap \spt \xi}\hspace*{-3mm} \phi^*(\abs{\ff }) +
  \phi^*(\abs{\pi}) \, d\bx\Big )\,.
\end{align*}
This and the previous estimate lead, as in Section~\ref{sec:Stokes},
to the inequality~\eqref{eq:tang-final-presp}.
\end{proof}
\begin{remark}
  \label{rem:tang}
  We can show with the same method that
  Proposition~\ref{prop:tangential2} holds with $\xi_P$ replaced by
  $\tilde \xi_P$ satisfying $(\ell1)$ with $R_P$ replaced by $R_P/2$,
  i.e.,~$\tilde \xi_P \in C^\infty _0(\widetilde \Omega_P)$, where
  $\widetilde \Omega_P:=\frac{1}{2}\Omega_P$ and
  $0\leq\tilde{\xi}_P\leq1$ with
  \begin{equation*}
    \chi_{\frac{1}{4}\Omega_P}(\bx)\leq\tilde{\xi}_P(\bx)\leq
    \chi_{\frac{3}{8}\Omega_P}(\bx).
  \end{equation*}
  %
\end{remark}
\begin{remark}
  The same procedure, with many simplifications, can be done in the
  interior of $\O$ for divided differences in all directions $\be^i$,
  $i=1,2,3$. By choosing
  $h\in\Big(0,\frac{1}{2}\dist(\spt\xi_{00},\partial\Omega)\Big)$ and
  mainly with the same steps as before this leads to
  \begin{equation}
    \label{inregp}
    \intO \xi_{00}^2 \bigabs{ \nabla \F (\bD\bu) }^2 + \phi (\xi_{00}
    \abs{ \nabla^2 \bu})\, d\bx    \le c(
    \norm{\xi_{00}}_{2,\infty}) \Big ( 1+ \intO  \phi^*(\abs{\ff})
    \, d\bx\Big )\,, 
  \end{equation}
and 
  \begin{align}
    \label{inreg-presp}
    \hspace*{-1mm}\intO \xi^2_{00} \bigabs{ \nabla \pi }^2 \, d\bx \le c(\delta^{p-2},
    \norm{\xi_{00}}_{2,\infty}) \le c(\delta^{p-2}, \norm{\xi_{00}}_{2,\infty})
    \Big ( 1\!+\! \intO  \phi^*(\abs{\ff}) \, d\bx\Big )\,, \hspace*{-3mm}
  \end{align}
  where $\xi_{00}$ is any cut-off function with compact support
  contained in $\Omega$.
\end{remark}
%
This remark proves the first estimate in Theorem~\ref{thm:MT2}.
Moreover, from~\eqref{inregp} and~\eqref{inreg-presp} we can infer immediately: 
\begin{corollary}
  \label{cor:loc}
  Under the same hypotheses of Theorem~\ref{thm:MT2} we obtain that
  $\bF(\bD)\in W^{1,2}_{\loc}(\Omega)$, $\bu\in W^{2,p}_\loc(\O)$, and
  $\pi \in W^{1,2}_\loc (\Omega)$. This implies --in particular-- that
  the equations of system~\eqref{eq} hold almost everywhere in
  $\Omega$
\end{corollary}
We observe that in the interior we can extract more information about
the local integrability of $\nabla ^2 \bu$, by using  simple arguments
combining H\"older and Sobolev inequalities, as for instance is done
in~\cite{bdr-7-5} in the space periodic case.  For every $G\subset
\subset \Omega$ and some $q\in ]p,2]$ we have
\begin{equation}
  \label{eq:imp}
  \begin{aligned}
    \norm{\nabla ^2 \bu}_{q,G}^q  &= \int_G \left (
      \varphi''(\abs{\bD\bu})\abs{\nabla ^2 \bu}^2\right)^{\frac q2}
    \big(\varphi''(\abs{\bD\bu})\big)^{-\frac q2} \, d\bx
    \\
    &\le \Big ( \int_G \varphi''(\abs{\bD\bu})\abs{\nabla ^2 \bu}^2 \,
    d\bx\Big )^{\frac q2} \Big ( \int_G \big(\delta
    +\abs{\bD\bu}\big)^{{(2-p)}\frac q{2-q}} \, d\bx\Big )^{\frac {2-q}2}
    \\
    &\le c\, \norm {\nabla \bF(\bD\bu)}_{2,G}^{q} \big (1 + \norm{\nabla
      \bu}^{2-p}_{(2-p)\frac q{2-q},G} \big )
    \\
    &\le c\, \norm {\nabla \bF(\bD\bu)}_{2,G}^{q} \big (1 + \norm{\nabla
      \bu}^{2-p}_{r,G}+ \norm{\nabla ^2 \bu}^{2-p}_{r,G} \big )\,,
  \end{aligned}
\end{equation}
where we used the H\"older inequality, the algebraic identity
$\partial^2_{jk}u^i= \partial_j\bD_{ik}+\partial_k\bD_{ij}-\partial_i\bD_{jk}$, 
and the embedding $W^{1,r}(G)\embedding L^{(2-p)\frac q{2-q}}(G)$
valid for $r=\frac{3q(2-p)}{6-q(p+1)}$. Requiring that $r= q$, this
yields $ q=\frac{3p}{p+1}$ and we can absorb the last term from the
right-hand side into the left-hand side of~\eqref{eq:imp}. Note that
$\norm{\nabla \bu}^{2-p}_{r,G}$ is finite since $\nabla \bu \in
L^{\frac{3p}{3-p}}(G)$. This and a Sobolev embedding theorem show
\begin{equation}
  \label{eq:reg2}
  \nabla ^2\bu \in L^{\frac{3p}{p+1}}_\loc(\Omega) \quad
  \text{and}\quad \nabla \bu
  \in L_\loc ^{3p}(\Omega)\,,
\end{equation}
which is the known best regularity results also in the space-periodic
case~\cite{bdr-7-5}.
\subsection{Regularity in the normal direction}
\label{sec:norp}
We follow the same reasoning as in the previous Section~\ref{sec:nor}
to prove the analogue of Proposition~\ref{prop:normal1}. Since the
results are very similar to those of the previous section, we now just
point out the differences going directly to the proof of the main
result.
\begin{proof}[Proof of Theorem~\ref{thm:MT2}]
  Let the same assumptions as in Proposition~\ref{prop:tangential2} be
  satisfied. Moreover, assume that $r_P$ in $(b3)$
  satisfies~\eqref{eq:smallness-a}. Then, we obtain for every
  $P\in\bou$ the following estimate (which is the counterpart
  of~\eqref{eq:estimate_normal_final} when the quadratic term is
  missing)
  \begin{equation}
    \label{eq:estimate_normal_final2}
    \begin{aligned}
      & \phi''(|\bD\bu|)|\nabla ^{2}\bu| \leq
      c\left(|\bff|+|\partial_{\tau}\pi|+
        \phi''(|\bD\bu|) 
        |\partial_{\tau}\nabla\bu|
      \right)\quad \textrm{a.e. in }\Omega_P.
    \end{aligned}
  \end{equation}
%
  In Section~\ref{sec:nor} we absorbed the last term
  in~\eqref{eq:estimate_normal_final2} in the term stemming from the
  extra stress tensor with $2$-structure. Here we have to proceed
  differently. Moreover, we have to deal with the problem that
  $\partial_\tau \bF(\bD\bu)$ in~\eqref{eq:tang-finalp} only yields a
  weighted information for $\partial _\tau \bD\bu$, while
  in~\eqref{eq:estimate_normal_final2} occurs $\partial_\tau
  \nabla\bu$.  The usual approach to resolve such problems by a Korn's
  inequality is not applicable in the moment, since it is not known
  whether or not an inequality of the type
  \begin{equation*}
    \int \phi''(|\bD\bu|)|\nabla \partial _\tau\bu|^2 \, d\bx \le c\,
    \int \phi''(|\bD\bu|)|\bD \partial _\tau\bu|^2 \, d\bx\,, 
  \end{equation*}
  holds true. Thus, we proceed differently and argue similarly
  to~\eqref{eq:imp}. Since we do not know if $\partial _3
  \bF(\bD\bu)\in L^2(\Omega)$ we use the anisotropic embedding from
  Theorem~\ref{thm:aniso}.  For technical reasons we work with the
  localization $\tilde \xi_P$ instead of $\xi_P$
  (cf.~Remark~\ref{rem:tang}) and use
  Proposition~\ref{prop:tangential2} for $\tilde{\xi}_P$.

  To have all following computations justified we use
  Corollary~\ref{cor:loc} and~\eqref{eq:reg2}, restrict ourselves to
  $q \in [1, \frac {3p}{p+1}\big ]$. We work, in the set
  $\widetilde{\Omega}_{\epsilon} = \widetilde{\Omega}_{P,\epsilon}$
  defined as follows
  \begin{equation*}
    \widetilde{\Omega}_{P,\epsilon}
    :=\{\bx\in \widetilde{\Omega}_P: \
    a_P(x)+\epsilon<x_3<a_P(x)+R'_P/2,\text{ for } 0<\epsilon<R_P/8 \}
%
  \end{equation*}
  for a fixed $P\in \bou$ and abbreviate $\tilde{\xi}=\tilde{\xi}_P$.
  (As usual we call $R_{P}:=\min\{R_{P},R'_{P}\}$.)
  Theorem~\ref{thm:aniso} yields
  \begin{align}
    \label{eq:troisi1}
    \begin{aligned}
      \norm{\tilde{\xi}\, \bF(\bD\bu)}_{3q,\Oe}&\le c\,\big (
      \norm{\partial _3 (\tilde{\xi}\, \bF(\bD\bu))}_{q,\Oe} +
      \norm{\partial _{\tau_1} (\tilde{\xi}\, \bF(\bD\bu))}_{q,\Oe}
      \\
      &\qquad\qquad+ \norm{\partial_{\tau_2} (\tilde{\xi}\,
        \bF(\bD\bu))}_{q,\Oe}\big)
      \\
      &\le c(\|\tilde{\xi}\, \|_{1,\infty})\,\big ( 1 +\norm{\tilde
        \xi\, \partial _3 \bF(\bD\bu)}_{q,\Oe} \big)\,,
    \end{aligned}
  \end{align}
  where we also used~\eqref{eq:7p}, the identity $\partial_\bv
  (\tilde{\xi}\, \bF(\bD\bu)) = \tilde{\xi}\,\partial_\bv \bF(\bD\bu)
  + \bF(\bD\bu)\,\partial _\bv \tilde{\xi}$, valid for any vector $\bv
  \in \setR^3$, and~\eqref{eq:tang-finalp}. To estimate the last term
  we recall that (cf.~\cite[Lemma 3.8]{bdr-7-5})
  \begin{equation}
    \sqrt{\phi''(|\bD\bu|)}|\nabla^2\bu|\sim|\nabla\bF(\bD\bu)|\,.\label{eq:nablaF}
  \end{equation}
  Thus, after multiplying both sides
  of~\eqref{eq:estimate_normal_final2} by $\tilde{\xi}$, integration
  over $\widetilde{\Omega}_\epsilon$ yields
  \begin{equation}
    \label{eq:Fq}
    \begin{aligned}
      &\int_{\widetilde \Omega_\epsilon} \tilde{\xi}^q |\nabla
      \bF(\bD\bu)|^q\, d\bx
      \\
      &\leq c\int_{\widetilde \Omega_\epsilon}
      \phi''(|\bD\bu|)^{-\frac{q}{2}}(|\bff|^q+|\partial_{\tau}\pi|^q)\,
      \tilde \xi^q + \phi''(|\bD\bu|)^{\frac{q}{2}}
      |\partial_{\tau}\nabla\bu|^q \, \tilde{\xi}^q\, d\bx\,.
    \end{aligned}
  \end{equation}
  We know that the right-hand side is finite, but depending on
  $\vep$. Now we derive estimates independent of $\vep$ for $q\in
  [1,\frac {3p}{p+1}\big ]$, as large as possible. By H\"older
  inequality with exponents $2/q$ and $2/(2-q)$ and~\eqref{eq:phi''}
  we obtain, for any $g\in L^2_\loc(\Omega)$,
  \begin{equation}
    \int_{\Omega_\epsilon} \phi''(|\bD\bu|)^{-\frac{q}{2}}
    |g|^q \tilde{\xi}^q\,d\bx\leq c \, \bignorm{g\, \tilde{\xi}}_{2,\widetilde\Omega_\epsilon}^{q}\,
    \bignorm{(\delta+|\bD\bu|)^{(2-p)}\tilde{\xi}}_{\frac q{2-q},\widetilde\Omega_\epsilon}^{ q}\,.
  \end{equation}
  To get a preliminary improvement of the integrability of $\bD\bu$ we
  apply this to $g=\ff$ and $g=\partial _\tau \pi$ for $q=p$. We also
  note that
  \begin{equation*}
    \int_{\widetilde\Omega_\epsilon}
    \phi''(|\bD\bu|)^{\frac{p}{2}}    |\partial_{\tau}\nabla\bu|^p
    \, \tilde{\xi}^p\,d\bx \leq  c\,\delta^{\frac{(p-2)p}{2}}\int_{\widetilde\Omega_\epsilon}
    \Big[    \phi(|\tilde\xi\,\partial_\tau\nabla\bu|)+\delta^p\Big]\,d\bx,
  \end{equation*}
  together with Proposition~\ref{prop:tangential2}, shows that for
  $q=p$ the right-hand side of~\eqref{eq:Fq} is finite and
  independent of $\vep$. Thus,~\eqref{eq:troisi1} and Levi's theorem
  on monotone convergence imply
  \begin{equation}
    \label{eq:first}
    \norm{\tilde{\xi}_P\, \bF(\bD\bu)}_{3p,\widetilde \Omega_P}\le c\,.
  \end{equation}
  Choosing a finite sub-covering of the covering $\bou \subset
  \bigcup_{P\in \bou} \widehat \Omega_P$, where $\widehat
  {\Omega}_P=\frac{1}{4}\Omega_P$ (cf.~$(b2)$), we get
  that~\eqref{eq:first} and Corollary~\ref{cor:loc} imply
  \begin{equation}
    \label{eq:normal-intermediate-step}
    \bF(\bD\bu)\in L^{3p}(\Omega)\quad\text{and}\quad
    \nabla\bu\in L^{\frac{3p^2}{2}}(\Omega)\,,
  \end{equation}
  where we also used~\eqref{eq:hammerq} and Korn's inequality.

  A better estimate of the last term on the right-hand side
  of~\eqref{eq:Fq} is the key to show that~\eqref{eq:Fq} is finite and
  independent of $\vep$, for some $q>p$.
  To this end we observe that, by using~\eqref{eq:two-derivatives-a}
  we have for $\alpha=1,2$
  \begin{equation*}
    \begin{aligned}
      \tilde{\xi}\,\partial_{\tau_\alpha}\nabla\bu&=\tilde{\xi}\,(\nabla\partial_{\tau_\alpha}\bu)-
      \tilde{\xi}\,(\nabla\partial_\alpha a\,\partial_3\bu)
      \\
      &=\nabla(\partial_{\tau_\alpha}\bu\,\tilde{\xi})-\partial_{\tau_\alpha}\bu\otimes\nabla\tilde{\xi}-
      \tilde{\xi}\,(\nabla\partial_\alpha a\,\partial_3\bu)\,.
    \end{aligned}
  \end{equation*}
  Consequently, 
  \begin{equation*}
    |\tilde{\xi}\,\partial_\tau\nabla\bu|^q\leq|\nabla(\partial_\tau\bu\,\tilde{\xi})|^q
    +c(\|\tilde{\xi}\|_{1,\infty},\|a\|_{C^{2,1}},q)|\nabla\bu|^q\qquad
    \textrm {a.e.\ in }\Omega_P\,,
  \end{equation*}
  and, in order to estimate the last term from the right hand side
  of~\eqref{eq:Fq}, we end up to consider the following integral
  \begin{equation*}
    \mathcal{I}:=  \int_{\widetilde{\Omega}_\epsilon}
    \phi''(|\bD\bu|)^{\frac{q}{2}}\Big(|\nabla(\partial_\tau\bu\,\tilde{\xi})|^q+|\nabla\bu|^q\Big)\,d\bx
    :=\mathcal{I}_1+\mathcal{I}_2\,.
  \end{equation*}
  Since $q\le \frac{3p}{p+1}\leq \frac{3p^2}{2}$, for all $p\in(1,2)$,
  we get from~\eqref{eq:normal-intermediate-step}
  \begin{equation*}
    \mathcal{I}_2 
    \leq  \delta^{\frac{(p-2)q}{2}}\int_{\widetilde{\Omega}_\epsilon} 
    |\nabla\bu|^{q}\,d\bx\leq c\,\Big (\int_{\Omega}|\nabla
    \bu|^{\frac{3p^2}{2}}\,d\bx\Big )^{\frac {2q}{3p^2}} < c(\delta^{-1})\,. 
  \end{equation*}
%
%
%
%
  Due to Korn's inequality (cf.~\cite[Thm.~5.17]{DieRS10}) there
  exists a constant depending on the John constant of $\widetilde
  \Omega$ and $q$ such that the term $\mathcal I_1$ can be estimated
  as follows:
  \begin{equation*}
    \begin{aligned}
      \mathcal{I}_1 
      &\leq \delta^{\frac{(p-2)q}{2}}\int_{\Omega_\epsilon}
      |\nabla(\partial_{\tau}\bu\,\tilde{\xi})|^q d\bx
      \\
      &\leq
      c\,\delta^{\frac{(p-2)q}{2}}\int_{\Omega_\epsilon}|\bD(\partial_{\tau}\bu\,\tilde{\xi})-
      \bigmean{\bD(\partial_{\tau}\bu\,\tilde{\xi})}_{\Omega_\epsilon}|^q
      +
      |\partial_{\tau}\bu\,\tilde{\xi}-
      \bigmean{\partial_{\tau}\bu\,\tilde{\xi}}_{\Omega_\epsilon}|^q\,d\bx
      \\
      &\le c\,\Big ( 1+
      \int_{\Omega_\epsilon}|\bD(\partial_{\tau}\bu\,\tilde{\xi})|^q
      \,d\bx\Big )\,,
    \end{aligned}
  \end{equation*}
  where we also used
  $\int_{\Omega_\epsilon}|\mean{g}_{\Omega_\epsilon}|^q\,d\bx\leq
  \int_{\Omega_\epsilon}|g|^q\,d\bx$, H\"older's inequality,
  and~\eqref{eq:normal-intermediate-step}.
  Using now the identities
  \begin{align*}
    \bD(\partial_{\tau_\alpha}\bu\,\tilde\xi)&=\tilde\xi\,\bD
    (\partial_{\tau_\alpha}\bu)+\partial_{\tau_\alpha}\bu\otimess\nabla\tilde\xi\,,
    \\
    \bD(\partial_{\tau_\alpha} \bu)&=\partial_{\tau_\alpha}
    \bD\bu+\bD(\nabla a)\otimess\partial_3\bu\,,
  \end{align*}
  we finally obtain
  \begin{equation*}
    \begin{aligned}
      \mathcal{I}_1
      \leq
      c(\delta^{-1},\norm{\tilde{\xi}}_{1,\infty},q,\norm{a}_{C^{2,1}},\widetilde
      \Omega_P) \Big(1+\int_{\Omega_\epsilon}
      |\partial_{\tau}\bD\bu|^q\,\tilde{\xi}^q\,d\bx\Big).
    \end{aligned}
  \end{equation*}
  From~\eqref{eq:3a} it follows
  \begin{equation}
    \label{eq:tangF} 
    \sqrt{\phi''(|\bD\bu|)}|\partial _\tau
    \bD\bu|\sim|\partial_\tau\bF(\bD\bu)|.
  \end{equation}
  Proceeding similarly to~\eqref{eq:imp} we get, also
  using~\eqref{eq:hammerq},
  \begin{equation*}
    \begin{aligned}
      \int_{\widetilde\Omega_\epsilon}|\partial_{\tau}\bD\bu|^q\,\tilde{\xi}^q\,d\bx
      &=
      \int_{\widetilde\Omega_\epsilon}\phi''(|\bD\bu|)^{\frac{q}{2}}
      |\partial_{\tau}\bD\bu|^q\,\tilde{\xi}^q\phi''(|\bD\bu|)^{-\frac{q}{2}}\,d\bx
      \\
      &\leq c \left[\int_{\widetilde{\Omega}_\epsilon}
        |\partial_{\tau}\bF(\bD\bu)|^2\,\tilde{\xi}^{\frac{4}{p'}}d\bx\right]^{\frac{q}{2}}
      \left[\int_{\widetilde{\Omega}_\epsilon}|\tilde{\xi}\,
        \bF(\bD\bu)|^{\frac{2-p}{p}\frac
          {2q}{2-q}}+\delta^{\frac{(2-p)q}{2-q}}
        \,d\bx\right]^{\frac{2-q}{2}}.
    \end{aligned}
  \end{equation*}
  Since we want to absorb the last term in the left-hand side
  of~\eqref{eq:troisi1}, we require that
  $\frac{2-p}{p}\frac{2q}{2-q}=3q$, which yields
  \begin{equation}
    \label{eq:q}
    q =\frac {8p-4}{3p}\,.
  \end{equation}
  Due to the fact that the exponent of the localization function of
  the first term on the right-hand side is not $2$ we enlarge the
  integration domain.  Recall, $\tilde{\xi}=\tilde{\xi}_P$ and by
  construction $\xi=\xi_P\equiv 1$ on
  $\widetilde{\Omega}=\widetilde{\Omega}_P$. Thus, we can write
  \begin{equation*}
    \begin{aligned}
      \int_{\widetilde{\Omega}_{P,\epsilon}}
      |\partial_{\tau}\bF(\bD\bu)|^2\,\tilde{\xi}_P^{\,\frac{4}{p'}}d\bx&\leq\int_{\widetilde{\Omega}_{P,\epsilon}}
      |\partial_{\tau}\bF(\bD\bu)|^2\,d\bx\leq
      \int_{\widetilde{\Omega}_P} |\partial_{\tau}\bF(\bD\bu)|^2\,d\bx
      \\
      &\leq \int_{\widetilde{\Omega}_P}
      |\partial_{\tau}\bF(\bD\bu)|^2\xi_P^2\,d\bx
      \leq\int_{{\Omega}_P}
      |\partial_{\tau}\bF(\bD\bu)|^2\xi^2_P\,d\bx\leq c,
    \end{aligned}
  \end{equation*}
  which is finite by Proposition~\ref{prop:tangential2}. Hence, we
  finally proved that
  \begin{equation}
    \label{eq:I}
    \mathcal{I}\leq
    c(\norm{\tilde{\xi}}_{1,\infty},\norm{a}_{C^{2,1}},\delta^{-1}) \Big(1+\norm{ 
      \bF(\bD\bu)\, \tilde{\xi}}^{q\frac{2-p}{p}}_{3q,\widetilde{\Omega}_\epsilon}\Big )\,,
  \end{equation}
  if $q$ is given by~\eqref{eq:q}.  We estimate the term with
  $\partial _\tau \pi$ in~\eqref{eq:Fq} as follows by using H\"older
  inequality
  \begin{align}
    \label{eq:other}
    \int_{\widetilde{\Omega}_\epsilon}\!
    \phi''(|\bD\bu|)^{-\frac{q}{2}}\tilde{\xi}^q|\partial_{\tau}\pi|^q\,
    \tilde{\xi}^q\,d\bx &\leq c\int_{\widetilde{\Omega}_\epsilon}
    \!\left(\delta^{\frac
        p2}\!+\!|\bF(\bD\bu)|\right)^{q\frac{2-p}{p}}
    |\partial_{\tau}\pi|^{q}\,\tilde{\xi}^{q}\,d\bx
    \\
    &\leq c \Big(1+\norm{ \bF(\bD\bu)\, \tilde{\xi}}^{q\frac{2-p}{p}}
    _{3q,\widetilde{\Omega}_\epsilon}\Big )
    \Big(\int_{\widetilde{\Omega}_\epsilon}
    |\partial_\tau\pi|^2\,\tilde{\xi}^{\frac{4}{p'}}\,d\bx\Big)^{\frac{3p}{4p-2}}
    \notag
    \\
    &\leq c \Big(1+\norm{ \bF(\bD\bu)\, \tilde{\xi}}^{q\frac{2-p}{p}}
    _{3q,\widetilde{\Omega}_\epsilon}\Big ) \Big(\int_{{\Omega}_P}
    |\partial_\tau\pi|^2\,{\xi}^{2}\,d\bx\Big)^{\frac{3p}{4p-2}}
    \notag
    \\
    &\leq c \Big(1+\norm{ \bF(\bD\bu)\, \tilde{\xi}}^{q\frac{2-p}{p}}
    _{3q,\widetilde{\Omega}_\epsilon}\Big )\,,\notag
  \end{align}
  where we increased the integration domain as above and used
  Proposition~\ref{prop:tangential2}. The term with $\ff$
  in~\eqref{eq:Fq} is estimated in the same way with several
  simplifications.  Inserting this estimate,~\eqref{eq:I}
  and~\eqref{eq:other} into~\eqref{eq:Fq}, we obtain
  from~\eqref{eq:troisi1}
  \begin{align*}
    \norm{ \bF(\bD\bu)\, \tilde{\xi}}_{3q,\Oe}&\le c\,\big ( 1
    +\norm{\tilde{\xi}\, \partial_3 \bF(\bD\bu)}_{q,\Oe} \big) \leq c
    \Big(1+\norm{ \bF(\bD\bu)\,
      \tilde{\xi}}^{\frac{2-p}{p}}_{3q,\widetilde{\Omega}_\epsilon}\Big)\,.
  \end{align*}
  Since $p \in (1,2)$ we can absorb the right-hand side into the
  left-hand side by Young's inequality. Thus, we proved for $q$
  satisfying~\eqref{eq:q}, that there exists a constant $c$ such
  that for every $P\in \bou$ there holds
  \begin{equation*}
    \norm{ \bF(\bD\bu) \tilde{\xi}_P}_{3q,\Oe}\! +\!\norm{\tilde
      \xi_P \partial _3 \bF(\bD\bu)}_{q,\Oe} \! \leq
    c(\norm{{\xi}_P}_{2,\infty},\norm{\tilde{\xi}_P}_{1,\infty},\norm{a_P}_{C^{2,1}},\delta^{-1})\,.  
  \end{equation*}
  Since $c$ is independent of $\epsilon>0$, Levi's monotone
  convergence theorem implies that
  \begin{equation}
    \label{eq:troisi2}
    \norm{ \bF(\bD\bu) \tilde{\xi}_P}_{\frac{8p-4}{p},\widetilde \Omega_P} +\norm{\tilde
      \xi_P \partial _3 \bF(\bD\bu)}_{\frac{8p-4}{3p},\widetilde \Omega_P} 
    \leq c    \, 
  \end{equation}
  which proves the third estimate in Theorem~\ref{thm:MT2}.
  Using the same covering argument which led
  to~\eqref{eq:normal-intermediate-step} we now obtain
  \begin{equation*}
    \nabla \bF(\bD\bu)\in
    L^{\frac{8p-4}{3p}}(\Omega) \quad\text{and}\quad   \bF(\bD\bu)\in
    L^{\frac{8p-4}{p}}(\Omega).
  \end{equation*}
  Then, by usual manipulation of the quantity $\bF(\bD\bu)$ we obtain
  that
  \begin{equation*}
    \nabla^2\bu\in L^{\frac{4p-2}{p+1}}(\Omega) \quad\text{and}\quad
    \nabla \bu\in  L^{{4p-2}}(\Omega)\,, 
  \end{equation*}
  or, by considering tangential and normal derivatives of $\nabla\bu$,
  the statements concerning the derivatives of $\bu$ in
  Theorem~\ref{thm:MT2}. This and~\eqref{eq:pi3} together
  with~\eqref{eq:tang-final-presp} proves the statements concerning
  the pressure $\pi$ in Theorem~\ref{thm:MT2}.  Theorem~\ref{thm:MT2}
  is now completely proved.
\end{proof}
\section*{Acknowledgments}
Both authors thank INDAM for the financial support. LCB is a member of
GNAMPA.
\def\cprime{$'$} \def\cprime{$'$} \def\cprime{$'$}
  \def\ocirc#1{\ifmmode\setbox0=\hbox{$#1$}\dimen0=\ht0 \advance\dimen0
  by1pt\rlap{\hbox to\wd0{\hss\raise\dimen0
  \hbox{\hskip.2em$\scriptscriptstyle\circ$}\hss}}#1\else {\accent"17 #1}\fi}
  \def\cprime{$'$} \def\polhk#1{\setbox0=\hbox{#1}{\ooalign{\hidewidth
  \lower1.5ex\hbox{`}\hidewidth\crcr\unhbox0}}}

\end{document}